\documentclass[11pt]{article}

\usepackage[centertags]{amsmath}
\usepackage{amsfonts}
\usepackage{amssymb}
\usepackage{amsthm}
\usepackage{newlfont}
\usepackage{bibspacing}
\usepackage{verbatim}

\newtheorem{thm}{Theorem}[section]
\newtheorem*{thm*}{Theorem}
\newtheorem{cor}[thm]{Corollary}

\newtheorem{lem}[thm]{Lemma}

\newtheorem{prop}[thm]{Proposition}
\newtheorem*{prop*}{Proposition}

\newtheorem*{conj*}{Conjecture}

\newtheorem*{dfn*}{Definition}
\theoremstyle{definition}
\newtheorem{rem}[thm]{\textbf{Remark}}
\newtheorem*{rmk*}{Remark}
\newtheorem*{fact*}{Fact}

\theoremstyle{proof}

\newcommand{\abs}[1]{\left\vert#1\right\vert}
\newcommand{\set}[1]{\left\{#1\right\}}
\newcommand{\brac}[1]{\left(#1\right)}
\newcommand{\scalar}[1]{\left \langle #1 \right \rangle}

\newcommand{\Real}{\mathbb{R}}

\newcommand{\eps}{\varepsilon}

\newcommand{\I}{\mathcal{I}}
\renewcommand{\H}{\mathcal{H}}
\newcommand{\J}{\mathcal{I}^\flat}
\newcommand{\g}{p}

\newlength{\defbaselineskip}
\newcommand{\setlinespacing}[1]           {\setlength{\baselineskip}{#1 \defbaselineskip}}

\numberwithin{equation}{section}

\begin{document}

\title{Sharp Isoperimetric Inequalities and Model Spaces for Curvature-Dimension-Diameter Condition}
\author{Emanuel Milman\textsuperscript{1}}
\date{}

\footnotetext[1]{Department of Mathematics,
Technion - Israel Institute of Technology, Haifa 32000, Israel. Supported by ISF, GIF and the Taub Foundation (Landau Fellow).
Email: emilman@tx.technion.ac.il.}

\maketitle

\begin{abstract}
We obtain new sharp isoperimetric inequalities on a Riemannian manifold equipped with a probability measure, whose generalized Ricci curvature is bounded from below (possibly negatively), and generalized dimension and diameter of the convex support are bounded from above (possibly infinitely). Our inequalities are \emph{sharp} for sets of any given measure and with respect to all parameters (curvature, dimension and diameter). Moreover, for each choice of parameters, we identify the \emph{model spaces} which are extremal for the isoperimetric problem. In particular, we recover the Gromov--L\'evy and Bakry--Ledoux isoperimetric inequalities, which state that whenever the curvature is strictly \emph{positively} bounded from below, these model spaces are the $n$-sphere and Gauss space, corresponding to generalized dimension being $n$ and $\infty$, respectively. In all other cases, which seem new even for the classical Riemannian-volume measure, it turns out that there is no \emph{single} model space to compare to, and that a simultaneous comparison to a natural \emph{one parameter family} of model spaces is required, nevertheless yielding a sharp result.
\end{abstract}

\section{Introduction}

Let $(M^n,g)$ denote an $n$-dimensional ($n \geq 2$) complete oriented smooth Riemannian manifold, and let $\mu$ denote a probability measure on $M$ having density $\Psi$ with respect to the Riemannian volume form $vol_g$.

\begin{dfn*}[Generalized Ricci Tensor] Given $q \in [0,\infty]$ and assuming that $\Psi > 0$ and $\log(\Psi) \in C^2$, we denote by $Ric_{g,\Psi,q}$ the following generalized Ricci tensor:
\begin{eqnarray}
\label{eq:Ric-tensor1}
Ric_{g,\Psi,q} & := & Ric_g - \nabla^2_g \log(\Psi) - \frac{1}{q} \nabla_g \log(\Psi) \otimes \nabla_g \log(\Psi) \\
\label{eq:Ric-tensor2} & = & Ric_g - q \frac{\nabla^2_g \Psi^{1/q}}{\Psi^{1/q}} ~.
\end{eqnarray}
When $q=\infty$, the last term in (\ref{eq:Ric-tensor1}) is interpreted as $0$, whereas when $q=0$, this term only makes sense if $\Psi$ is constant, in which case $Ric_{g,\Psi,0} := Ric_g$. Here as usual $Ric_g$ denotes the Ricci curvature tensor and $\nabla_g$ denotes the Levi-Civita covariant derivative.
\end{dfn*}

\begin{dfn*}[Curvature-Dimension-Diameter Condition]
$(M^n,g,\mu)$ is said to satisfy the Curvature-Dimension-Diameter Condition $CDD(\rho,n+q,D)$ ($\rho \in \Real$, $q \in [0,\infty]$, $D \in (0,\infty]$), if $\mu$ is supported on the closure of a geodesically convex domain $\Omega \subset M$ of diameter at most $D$, having (possibly empty) $C^2$ boundary, $\mu = \Psi \cdot vol_g|_\Omega$ with $\Psi > 0$ on $\overline{\Omega}$ and $\log(\Psi) \in C^2(\overline{\Omega})$, and as $2$-tensor fields:
\[
Ric_{g,\Psi,q} \geq \rho g \; \text{ on } \Omega ~.
\]
\end{dfn*}

When $\Omega = M$ and $D = +\infty$, the latter definition coincides with the celebrated Bakry--\'Emery Curvature-Dimension condition $CD(\rho,n+q)$, introduced in an equivalent form in \cite{BakryEmery} (in the more abstract framework of diffusion generators). Indeed, the generalized Ricci tensor incorporates information on curvature and dimension from both the geometry of $(M,g)$ and the measure $\mu$, and so $\rho$ may be thought of as a generalized-curvature lower bound, and $n+q$ as a generalized-dimension upper bound. The generalized Ricci tensor (\ref{eq:Ric-tensor1}) was introduced with $q=\infty$ in \cite{Lichnerowicz1970GenRicciTensorCRAS,Lichnerowicz1970GenRicciTensor} and in general in \cite{BakryStFlour} (the equivalent form (\ref{eq:Ric-tensor2}) was noted in \cite{LottRicciTensorProperties}), and has been extensively studied and used in recent years (see e.g. also \cite{QianWeightedVolumeThms,LedouxLectureNotesOnDiffusion,VonRenesseSturmRicciChar,PerelmanEntropyFormulaForRicciFlow,BakryQianGenRicComparisonThms,SturmCD12,
LottVillaniGeneralizedRicci,WeiWylie-GenRicciTensor,MorganBook4Ed} and the references therein).

\medskip

In this work, we obtain a \emph{sharp} isoperimetric inequality on $(M^n,g,\mu)$ under the $CDD(\rho,n+q,D)$ condition, for the \emph{entire} range of parameters $\rho \in \Real$, $q \in [0,\infty]$, $D \in (0,\infty]$, in a single unified framework. In particular, for each choice of parameters, we identify the \emph{model spaces} which are extremal for the isoperimetric problem. Our results seem new even in the classical constant-density case ($q=0$) when $\rho \leq 0$ and $D < \infty$ or when $\rho > 0$ and $D < \pi \sqrt{(n-1) / \rho}$.
We start by recalling the notion of an isoperimetric inequality in a general measure-metric space setting and some previously known results.

\subsection{Isoperimetric Inequalities}

Let $(\Omega,d)$ denote a separable metric space, and let $\mu$ denote a Borel probability measure on $(\Omega,d)$. The Minkowski (exterior) boundary measure $\mu^+(A)$ of a Borel set $A \subset \Omega$ is defined as $\mu^+(A) := \liminf_{\eps \to 0} \frac{\mu(A^d_{\eps}) -
\mu(A)}{\eps}$, where $A_{\eps}=A^d_{\eps} := \set{x \in \Omega ; \exists y
\in A \;\; d(x,y) < \eps}$ denotes the $\eps$ extension of $A$ with
respect to the metric $d$. The isoperimetric profile $\I =
\I(\Omega,d,\mu)$ is defined as the pointwise maximal function $\I
: [0,1] \rightarrow \Real_+ \cup \set{+\infty}$, so that $\mu^+(A) \geq \I(\mu(A))$, for
all Borel sets $A \subset \Omega$. An isoperimetric inequality measures the relation between the boundary measure and the measure of a set, by providing a lower bound on $\I(\Omega,d,\mu)$ by some (non-trivial) function $I: [0,1] \rightarrow \Real_+$. In our manifold-with-density setting, we will always assume that the metric $d$ is given by the induced geodesic distance on $(M,g)$, and write $\I = \I(M,g,\mu)$.

When $(\Omega,d) = (\Real,|\cdot|)$, we also define $\J  = \J(\Real,\abs{\cdot},\mu)$ as the pointwise maximal function $\J : [0,1] \rightarrow \Real_+ \cup \set{+\infty}$, so that $\mu^+(A) \geq \J(\mu(A))$ for all half lines $A = (-\infty,a)$ and $A = (a,\infty)$ (the difference with the function $\I$ being that the latter is tested on arbitrary Borel sets $A$). Obviously $\J \geq \I$, and a result of S. Bobkov \cite[Proposition 2.1]{BobkovExtremalHalfSpaces} asserts that $\J = \I$ when $\mu = f(x) dx$ and $f$ is \emph{log-concave}, meaning that $- \log(f) : \Real \rightarrow \Real \cup \set{+\infty}$ is convex.

\medskip

When $\rho > 0$, sharp isoperimetric inequalities under the $CD(\rho,n+q)$ condition are known and well understood, thanks to the existence of comparison model spaces on which equality is attained.
The first such result was obtained by M. Gromov in \cite{GromovGeneralizationOfLevy} (reprinted in \cite[Appendix C]{Gromov}), extending P. L\'evy's isoperimetric inequality on the sphere \cite{LevyIsopInqOnSphere,SchmidtIsopOnModelSpaces}, in the constant density case ($q=0$). Setting $\mu_g = vol_g / vol_g(M)$, the Gromov--L\'evy isoperimetric inequality states that if $Ric_g \geq \rho g$ with $\rho > 0$, then $\I(M,g,\mu_g) \geq \J(\Real,\abs{\cdot},\mu_{n,\rho})$, where $\mu_{n,\rho}$ denotes the probability measure supported on $[0,\pi \sqrt{(n-1)/\rho}]$ with density proportional to $\sin(\sqrt{\rho/(n-1)}t)^{n-1}$. In particular, by testing geodesic balls on $(S^n,g^\rho_{can})$, the $n$-dimensional sphere with Ricci curvature equal to $\rho$,
it follows that $\I(S^n,g^\rho_{can},\mu_{g^\rho_{can}}) = \J(\Real,\abs{\cdot},\mu_{n,\rho})$, recovering the classical isoperimetric inequality on the sphere. The case when $q=+\infty$ was treated by Bakry and Ledoux \cite{BakryLedoux} (see also Morgan \cite{MorganManifoldsWithDensity} for a geometric derivation), who showed that if $(M,g,\mu)$ satisfies the $CD(\rho,\infty)$ condition with $\rho > 0$, then $\I(M,g,\mu) \geq \J(\Real,\abs{\cdot},\gamma^\rho_1)$, where $\gamma^\rho_k$ denotes the standard Gaussian density on $\Real^k$ with covariance matrix $\rho^{-1} Id$, and $\abs{\cdot}$ denotes the standard Euclidean metric on $\Real^k$. In particular, this recovers the isoperimetric inequality of Sudakov and Tsirelson \cite{SudakovTsirelson} and independently Borell \cite{Borell-GaussianIsoperimetry}, stating that $\I(\Real^n,\abs{\cdot},\gamma^\rho_n) = \J(\Real,\abs{\cdot},\gamma^\rho_1)$. An extension of these results to $q \in (0,\infty)$ when $\rho > 0$ was subsequently obtained by Bayle in \cite[Appendix E]{BayleThesis}. 

When $\rho \leq 0$, the situation is very different, and without requiring some \emph{additional} information on the space $(M,g,\mu)$, no isoperimetric inequality can be deduced under the $CD(\rho,n+q)$ condition (in the sense that $\I(M,g,\mu)$ can be arbitrarily small). Various types of information have been considered in the literature. In \cite{BuserReverseCheeger}, Buser considered the existence of a spectral-gap in the constant-density ($q=0$) case; this was later extended to the $q=\infty$ case by Ledoux \cite{LedouxSpectralGapAndGeometry}, and generalized to other Sobolev type inequalities (e.g. \cite{BakryLedoux,LedouxSpectralGapAndGeometry,EMilmanRoleOfConvexityInFunctionalInqs}). Various authors (e.g. \cite{WangImprovedIntegrabilityForLogSob,BobkovGaussianIsoLogSobEquivalent,
BartheIntegrabilityImpliesIsoperimetryLikeBobkov,BartheKolesnikov,EMilmanGeometricApproachPartI}) considered an integrability condition of the form $\int_M \exp(\beta(d(x,x_0))) d\mu(x) < \infty$ for some (any) fixed $x_0 \in M$.
In \cite{EMilman-RoleOfConvexity,EMilmanGeometricApproachPartI}, we considered concentration inequalities, and showed that under the $CD(\rho,\infty)$ condition, these imply isoperimetric inequalities which are essentially best possible, up to dimension independent constants. But perhaps the most classical assumption from the view point of Riemannian Geometry is an upper bound on the diameter, which is a particular case of the integrability and concentration assumptions mentioned above. By considering domains $\Omega$ with bottlenecks, it is immediate to see that again no isoperimetric inequality can be deduced in general, and so requiring that $\Omega$ be geodesically convex (see Section \ref{sec:pre} for a precise definition) is a natural assumption; furthermore, this amounts to the natural requirement that the metric space $(\Omega,d)$, where $d$ is the induced geodesic distance on $(M,g)$, be a geodesic space. We thus arrive at the $CDD(\rho,n+q,D)$ condition.

Various isoperimetric inequalities assuming $CDD(\rho,n+q,D)$ with $\Omega = M$ and $D < \infty$ have been obtain for the classical constant density case $q=0$ in \cite{CrokeIsoperimetricBounds,BBGimprovingGromov,GallotIsoperimetricInqs}. In particular, when $\rho > 0$ and $D < \pi \sqrt{(n-1) / \rho}$, Croke \cite{CrokeEigenvaluePinching} and B\'erard--Besson--Gallot \cite{BBGimprovingGromov} obtained improvements over the Gromov--L\'evy inequality. Some of these results were extended to $q > 0$ by Bayle in \cite{BayleThesis}.

\medskip

However, with the exception of the known results under the $CD(\rho,n+q)$ condition when $\rho > 0$ (and $D=\infty$), none of the above mentioned results yield \emph{sharp} isoperimetric inequalities for all $v\in(0,1)$. Moreover, most known results fail to capture the behavior of $\I(v)$ for $v \in (0,1/2]$ close to and away from $0$ simultaneously, and miss the optimal inequality by dimension dependent factors. The difficulty when $\rho \leq 0$ lies in that there does not seem to be a \emph{good} model space to compare to, as in the Gromov--L\'evy or Bakry--Ledoux results. The purpose of this work is to fill this gap, providing a sharp isoperimetric inequality under the $CDD(\rho,n+q,D)$ condition in the entire range $\rho \in \Real$, $q \in [0,\infty]$, $D \in (0,\infty]$ and $v \in (0,1)$.

\subsection{Results}

Given $\delta \in \Real$, set as usual:
\[
\begin{array}{ccc}
 s_\delta(t) := \begin{cases}
\sin(\sqrt{\delta} t)/\sqrt{\delta} & \delta > 0 \\
t & \delta = 0 \\
\sinh(\sqrt{-\delta} t)/\sqrt{-\delta} & \delta < 0
\end{cases}

& , &

 c_\delta(t) := \begin{cases}
\cos(\sqrt{\delta} t) & \delta > 0 \\
1 & \delta = 0 \\
\cosh(\sqrt{-\delta} t) & \delta < 0
\end{cases}
\end{array} ~.
\]

Given a continuous function $f : \Real \rightarrow \Real$ with $f(0) \geq 0$, we denote by $f_+ : \Real \rightarrow \Real_+$ the function coinciding with $f$ between its first non-positive and first positive roots, and vanishing everywhere else, i.e. $f_+ := f 1_{[\xi_{-},\xi_{+}]}$ with $\xi_{-} = \sup\set{\xi \leq 0; f(\xi) = 0}$ and $\xi_{+} = \inf\set{\xi > 0; f(\xi) = 0}$.

\begin{dfn*}
Given $H,\rho \in \Real$ and $m \in [0,\infty]$, set $\delta := \rho / m$ if $m > 0$ and define the following (Jacobian) function of $t \in \Real$:
\[
J_{H,\rho,m}(t) :=
\begin{cases}
1_{\set{t=0}}  & m = 0 , \rho > 0 \\
1_{\set{H t \geq 0}} & m = 0 , \rho \leq 0 \\
\brac{c_\delta(t) + \frac{H}{m} s_\delta(t)}_+^{m} & m \in (0,\infty) \\
\exp(H t - \frac{\rho}{2} t^2) & m = \infty
\end{cases} ~.
\]
\end{dfn*}

\begin{rem} \label{rem:J-char}
Observe that since $c_\delta(t) = 1 - \frac{\delta}{2} t^2 + o(\delta)$ and $s_\delta(t) = t + o(\delta)$ as $\delta \rightarrow 0$, it follows that $\lim_{m \rightarrow \infty} J_{H,\rho,m} =  J_{H,\rho,\infty}$. A direct calculation also verifies that $\lim_{m \rightarrow 0+} J_{H,\rho,m} =  J_{H,\rho,0}$. Also observe that when $m > 0$ (and with the usual interpretation when $m=\infty$), $J_{H,\rho,m}$ coincides with the solution $J$ to the following second order ODE, on the maximal interval containing the origin where such a solution exists:
\[
- (\log J)''  - \frac{1}{m} ((\log J)')^2 = - m \frac{(J^{1/m})''}{J^{1/m}} = \rho ~,~ J(0) = 1 ~,~ J'(0) = H ~.
\]
The connection to (\ref{eq:Ric-tensor1}) and (\ref{eq:Ric-tensor2}) is evident.
\end{rem}

Lastly, given a non-negative integrable function $f$ on a closed interval $L \subset \Real$, we denote for short $\I(f,L) := \I(\Real,\abs{\cdot},\mu_{f,L})$, where $\mu_{f,L}$ is the probability measure supported in $L$ with density proportional to $f$ there. Similarly, we set $\J(f,L) := \J(\Real,\abs{\cdot},\mu_{f,L})$. When $\int_L f(x) dx = 0$ we set $\J(f,L) = \I(f,L) \equiv +\infty$, and when $\int_L f(x) dx = +\infty$ we set $\J(f,L) = \I(f,L) \equiv 0$.

\begin{thm} \label{thm:main1}
Let $(M^n,g,\mu)$ satisfy the $CDD(\rho,n+q,D)$ condition with $\rho \in \Real$, $q \in [0,\infty]$ and $D \in (0,+\infty]$. Then:
\begin{equation} \label{eq:main1}
 \I(M,g,\mu) \geq \inf_{H \in \Real, a \in [D-D,D]} \J\brac{ J_{H,\rho,n+q-1}, [-a,D-a] } ~,
\end{equation}
where the infimum is interpreted pointwise on $[0,1]$.
\end{thm}
\begin{rem}
We employ throughout the convention $\infty-\infty =  -\infty+\infty=\infty$ and $[-\infty,\infty] = \Real$.
\end{rem}

In fact, the $\J$ above may be replaced by $\I$, leading to the same lower bound (see Corollary \ref{cor:replace-J-with-I}), and the infimum above is actually always attained (see Corollary \ref{cor:inf-attained}). The bound (\ref{eq:main1}) was deliberately formulated to cover the entire range of values for $\rho$, $n$, $q$ and $D$ simultaneously, indicating its universal character, but it may be easily simplified as follows (the elementary proof is deferred to Section \ref{sec:model}):

\begin{cor} \label{cor:main1}
Under the same assumptions and notation as in Theorem \ref{thm:main1}, and setting $\delta:= \frac{\rho}{n+q-1}$, we have:
\begin{description}
 \item[Case 1 - $q < \infty$, $\rho > 0$, $D<\pi / \sqrt{\delta}$:]
\[
\I(M^n,g,\mu) \geq \inf_{\xi \in [0,\pi/\sqrt{\delta}-D]} \J\brac{\sin(\sqrt{\delta} t)^{n+q-1}, [\xi,\xi +D] } ~.
\]
\item[Case 2 - $q < \infty$, $\rho > 0$, $D \geq \pi / \sqrt{\delta}$:]
\[
\I(M^n,g,\mu) \geq \J\brac{ \sin(\sqrt{\delta} t)^{n+q-1} , [0,\pi/\sqrt{\delta}] } ~.
\]
\item[Case 3 - $q < \infty$, $\rho = 0$, $D<\infty$:]
\begin{eqnarray*}
\!\!\!\!\!\!\!\!\!\!\! \I(M^n,g,\mu)(v) &\geq& 
\min \left \{ \begin{array}{l}  \inf_{\xi \geq 0} \J( t^{n+q-1} , [\xi,\xi+D] )(v) ~,\\
\phantom{\inf_{\xi \in \Real}} \J(1,[0,D])(v)
\end{array}
\right \} \\
 & = & \frac{n+q}{D} \inf_{\xi \geq 0}  \frac{\brac{\min(v,1-v) (\xi+1)^{n+q} + \max(v,1-v) \xi^{n+q}}^{\frac{n+q-1}{n+q}}}{(\xi+1)^{n+q} - \xi^{n+q}}  \;\;\; \forall v \in [0,1] ~.
\end{eqnarray*}
\item[Case 4 - $q < \infty$, $\rho < 0$, $D<\infty$:]
\[
 \I(M^n,g,\mu) \geq \min \left \{ \begin{array}{l}
\inf_{\xi \geq 0} \J( \sinh(\sqrt{-\delta} t)^{n+q-1}, [\xi,\xi+D] ) ~ , \\
\phantom{\inf_{\xi \in \Real}} \J( \exp(\sqrt{-\delta} (n+q-1) t) , [0,D] ) ~, \\
\inf_{\xi \in \Real} \J(\cosh(\sqrt{-\delta} t)^{n+q-1},[\xi,\xi+D] )
\end{array}
\right \} ~.
\]
\item[Case 5 - $q = \infty$, $\rho \neq 0$, $D < \infty$:]
\[
\I(M^n,g,\mu) \geq \inf_{\xi \in \Real} \J(\exp(-\frac{\rho}{2} t^2),[\xi,\xi+D]) ~.
\]
\item[Case 6 - $q = \infty$, $\rho > 0$, $D = \infty$:]
\[
\I(M^n,g,\mu) \geq \J(\exp(-\frac{\rho}{2} t^2),\Real) = \J(\Real,\abs{\cdot},\gamma^\rho_1) ~.
\]
\item[Case 7 - $q = \infty$, $\rho = 0$, $D<\infty$:]
\begin{eqnarray*}
\I(M^n,g,\mu)(v) &\geq& \inf_{H \geq 0} \J(\exp(H t),[0,D])(v) \\
&=& \frac{1}{D} \inf_{w > 0} (\min(v,1-v) + w) \log(1+1/w) \;\;\; \forall v \in [0,1] ~.
\end{eqnarray*}
\end{description}
In all the remaining cases, we have the trivial bound $\I(M^n,g,\mu) \geq 0$.
\end{cor}

Note that when $q$ is an integer, $\J( \sin(\sqrt{\delta} t)^{n+q-1} , [0,\pi/\sqrt{\delta}] ) = \I(S^{n+q},g^\rho_{can},\mu_{g^\rho_{can}})$ by the isoperimetric inequality on the sphere, and so Case 2 with $q=0$ recovers the Gromov--L\'evy isoperimetric inequality \cite{GromovGeneralizationOfLevy} stated earlier; for general $q < \infty$, Case 2 was obtained by Bayle \cite[Theorem 3.4.18]{BayleThesis}.
Case 6 recovers the Bakry--Ledoux isoperimetric inequality \cite{BakryLedoux,MorganManifoldsWithDensity}. To the best of our knowledge, all remaining cases are new.
A non-sharp version of Case 7 (with a strictly worse numerical constant) may also be deduced from our results in \cite{EMilmanGeometricApproachPartI}.
To illuminate the transition between Cases 1 and 2, note that if $(M^n,g,\mu)$ satisfies the $CD(\rho,n+q)$ condition with $\rho > 0$, the diameter of $M$ is bounded above by $\pi / \sqrt{\delta}$: when $q=0$ this is the classical Bonnet-Myers theorem (e.g. \cite{GHLBookEdition3}), which was extended to $q > 0$ by Qian \cite{QianWeightedVolumeThms}; these bounds also easily follow from our proof.

The main justification for considering the bounds given in Theorem \ref{thm:main1} and Corollary \ref{cor:main1} is:

\begin{thm} \label{thm:main2}
For any $n \geq 2$, $\rho \in \Real$, $q \in [0,\infty]$, $D \in (0,\infty]$ and $v \in [0,1]$, the lower bound provided in Corollary \ref{cor:main1} (or equivalently, the one provided in Theorem \ref{thm:main1}) on $\I(M,g,\mu)(v) = \inf \set{ \mu^+(A) \; ;\;  A \subset M \; ,\; \mu(A) = v }$ for a manifold-with-density $(M^n,g,\mu)$ satisfying the $CDD(\rho,n+q,D)$ condition, is sharp.
\end{thm}

We conclude that with the exception of the previously known cases 2 and 6 above, there is no \emph{single} model space to compare to, and that a simultaneous comparison to a natural \emph{one parameter family} of model spaces is required, nevertheless yielding a sharp comparison result. The fact that the sharp lower bound on the boundary measure of a set having measure $v \in (0,1)$ is determined by a model space depending not only on $\rho$, $n+q$ and $D$, but also (in general) on $v$, was (to the best of our knowledge) unanticipated (see also Subsection \ref{subsec:BakryQian}). 

\medskip

Note that Theorem \ref{thm:main2} would hold trivially if the requirement that the bounds are sharp for \emph{any $n \geq 2$} were omitted from its formulation, and if we extend our definitions to include the case of one-dimensional manifolds-with-density:
\begin{dfn*}
The one-dimensional space $(\Real,|\cdot|,\mu)$ is said to satisfy the $CDD(\rho,1+q,D)$ condition, if there exists
an open interval $\Omega \subset (\Real,|\cdot|)$ of length at most $D$, whose closure supports a probability measure $\mu = \Psi(x) dx$ with $\Psi > 0$ in $\Omega$ and $\log(\Psi) \in C^2(\Omega)$, so that:
\[
- (\log \Psi)''  - \frac{1}{q} ((\log \Psi)')^2 = - q \frac{(\Psi^{1/q})''}{\Psi^{1/q}} \geq \rho \; \text{ in } \Omega
\]
(with the usual interpretation when $q=0$ or $q=\infty$).
\end{dfn*}

It is not hard to check (see Corollary \ref{cor:main1-1D}) that Theorem \ref{thm:main1} remains valid for such one-dimensional spaces. 
By construction, all of the one-dimensional model spaces given in Theorem \ref{thm:main1} (or equivalently Corollary \ref{cor:main1}) satisfy the $CDD(\rho,n+q,D)$ condition, immediately implying the sharpness in the (topological) one-dimensional case. 
It is not uncommon in the manifold-with-density literature to only demonstrate the optimality of a given estimate, as a function of the generalized dimension $n+q$, just for the (topological) one-dimensional case $n=1$; however,  we insist on demonstrating the optimality for all $n \geq 2$ as well, and this poses a greater technical challenge. 

\begin{rem} \label{rem:1D}
Note that in the one-dimensional case we do not require that $\Psi > 0$ nor $\log(\Psi) \in C^2$ on the entire $\overline{\Omega}$, as we did for technical reasons in the higher-dimensional case.
To dispose of this and some of our other technical assumptions, we present an appropriate approximation argument in Section \ref{sec:gen}.
\end{rem}

\subsection{Method}

Our method is entirely geometric, following the approach set forth by Gromov in \cite{GromovGeneralizationOfLevy}. We heavily rely on results from Geometric Measure Theory asserting the regularity of isoperimetric minimizers, both in the interior and on the boundary. To estimate the measure swept out by the normal map emanating from the regular part of the minimizer under the $CDD(\rho,n+q,D)$ condition, we employ a generalized version of the Heintze--Karcher theorem due to V. Bayle \cite[Appendix E]{BayleThesis} and F. Morgan \cite{MorganManifoldsWithDensity}. This reduction to the one-dimensional case allows us to obtain the lower bound on the isoperimetric profile given by Theorem \ref{thm:main1}, without compromising on its sharpness.

\medskip

To prove the sharpness for any $n \geq 2$,  we emulate our one-dimensional model densities on a geodesically convex domain of an $n$-dimensional manifold, by thickening arbitrarily slightly in $n-1$ dimensions. When $\rho = 0$ or $q=\infty$ this is very easy to accomplish simply by considering Euclidean space, and so for instance the model spaces for Case 3 are truncated cones $\set{(x_1,y) \in \Real \times \Real^{n-1} ; x_1 \in [\xi_0,\xi_0+D_\eps], |y| \leq \eps x_1}$ ($\xi_0 \geq 0$) endowed with a density proportional to $x_1^{q}$, and for Case 5 these are rectangles of the form $[\xi_0,\xi_0 + D_\eps] \times [0,\eps]^{n-1}$ ($\xi_0 \in \Real$) endowed with a density proportional to $\exp(-\frac{\rho}{2} |x|^2)$ (or more precisely, smoothed versions thereof). However, to establish the sharpness when $\rho \neq 0$ and $q < \infty$, we already need to construct a family of rotationally-invariant manifolds endowed with appropriate metrics and densities, and this poses a much greater technical challenge, in part due to the required geodesic convexity of $\Omega$; in fact, the hardest case turns out to be the two-dimensional one.

\medskip

Applications of these results will be developed in a subsequent work. These include analysis of the asymptotic behaviour of the lower bounds given by Corollary \ref{cor:main1} as a function of the parameters $\rho$, $n+q$, $D$ and $v$, and a derivation of corresponding Sobolev inequalities on spaces satisfying the $CDD(\rho,n+q,D)$ condition, improving in many cases the best known bounds (see Subsection \ref{subsec:applications})

\medskip
The rest of this work is organized as follows. In Section \ref{sec:pre} we describe the ingredients from Riemannian Geometry and Geometric Measure Theory we require for the proof. Theorem \ref{thm:main1} and some generalizations are proved in Section \ref{sec:proof}. Corollary \ref{cor:main1} is deduced in Section \ref{sec:model}, where we identify the corresponding families of model densities. Theorem \ref{thm:main2} regarding the sharpness of our results is proved in Section \ref{sec:sharp}. An extension of the Curvature-Dimension-Condition is described in Section \ref{sec:gen}. Concluding Remarks are presented in Section \ref{sec:remarks}. Several useful properties of the model densities are collected in the Appendix.

\medskip
\noindent
\textbf{Acknowledgements.} I would like to thank Frank Morgan for his interest, suggestions and encouragement, and for patiently answering my questions. I also thank Shahar Mendelson for his comments regarding this work.

\section{Geometric Preliminaries} \label{sec:pre}

\subsection{Generalized Heintze--Karcher Theorem}

The first ingredient we will need is a generalization of the Heintze--Karcher theorem (\cite{HeintzeKarcher},\cite[Theorem 4.21]{GHLBookEdition3}), which is a classical volume comparison theorem in Riemannian Geometry when there is no density present. Given a $C^2$ hypersurface $S$ in $(M^n,g)$ oriented by a unit normal vector field $\nu$, the classical theorem bounds the volume of the one-sided neighborhood of $S$ in terms of the mean-curvature of $S$ and a lower bound on $Ric_g$. Recall that the \emph{mean-curvature} of $S$ at $x$, denoted $H^\nu_S(x)$, is defined as the trace of the second fundamental form $II^\nu_{S,x}$; it is customary to divide the trace by $n-1$, the dimension of $S$, but we will \emph{refrain} from this normalization here.
We conform to the following \emph{non-standard} convention for specifying the sign of $II^\nu_{S,x}$: the second fundamental form of the sphere in Euclidean space with respect to the \emph{outward} normal is \emph{positive} definite (formally: $II^\nu_{S,x}(u,v) = g(\nabla_u \nu , v)$ for $u,v \in T_x S$, where $\nabla$ is the covariant derivative). In the case that $(M,g)$ is equipped with a measure $\mu = \Psi \cdot vol_g$ with $\log \Psi \in C^1(M)$, we define following V. Bayle \cite{BayleThesis}:

\begin{dfn*}
The generalized mean-curvature of $S$ at $x \in S$ with respect to the measure $\mu$ and unit normal vector field $\nu$, denoted $H_{S,\mu}^\nu(x)$, is defined as:
\[
H_{S,\mu}^\nu(x) := H_S^\nu(x) + \nu(\log \Psi) (x) ~.
\]
\end{dfn*}

The following generalization of the classical Heintze--Karcher theorem (the case $q=0$) to the case of manifolds-with-density is due to V. Bayle \cite[Appendix E]{BayleThesis} when $q \in (0,\infty)$, and to F. Morgan \cite{MorganManifoldsWithDensity} in the case $q=\infty$ (the latter may also be obtained by a limiting argument in view of Remark \ref{rem:J-char}):

\begin{thm}[Generalized Heintze--Karcher, Bayle--Morgan] \label{thm:gen-HK}
Let $S$ denote a $C^2$ oriented hypersurface in an $n$-dimensional manifold $(M,g)$ with normal unit vector field $\nu$, and given $r > 0$, set:
\[
S_r^+ := \set{\exp_x(t \nu(x)) ; x \in S , t \in [0,r]} ~.
\]
Assume that for some $\rho \in \Real$ and $q \in [0,\infty]$:
\[
Ric_{g,\Psi,q} \geq \rho g \;\;\; \text{ on } S_r^+ ~.
\]
Then:
\[
\mu(S_r^+) \leq \int_S \int_0^{r} J_{H^\nu_{S,\mu}(x),\rho,n+q-1}(t) dt \; dvol_{S,\mu}(x) ~,
\]
where $vol_{S,\mu} = \Psi \cdot vol_S$, and $vol_S$ denotes the induced Riemannian volume form on $S$.
\end{thm}

\begin{rem} \label{rem:first-variation}
It is easy to check (see \cite[3.4.6]{BayleThesis}, \cite[Proposition 7]{MorganManifoldsWithDensity}) that the first variation $\delta^1(u)$ of $vol_{S,\mu}(S)$ by
a normal variation of compact support and constant velocity $u(x)$ along $\nu$, is precisely determined by the generalized mean-curvature: $\delta^1(u) = \int_S H^\nu_{S,\mu}(x) u(x) dvol_{S,\mu}(x)$. This extends the classical fact from Riemannian geometry in the case of constant density (e.g. \cite[Theorem 5.20]{GHLBookEdition3}).
\end{rem}

\subsection{Existence and Regularity of Isoperimetric Minimizers}

The second ingredient we will need is the existence and regularity theory of isoperimetric minimizers on manifolds-with-density, provided by Geometric Measure Theory; for an extensive introduction to the latter, we refer to \cite{MorganBook4Ed,FedererBook,GiustiBook}. The results we describe below are classical in the case that $\Omega$ is a domain in Euclidean space with constant density, but the adaptations to the manifold-with-density setting are not as well known. We therefore sketch the argument where it is possible, and provide references elsewhere.

An isoperimetric minimizer in $(\Omega,d,\mu)$ of given measure $v \in (0,1)$ is a Borel set $A \subset \Omega$ with $\mu(A) = v$ for which the following infimum is attained:
\[
\mu^+(A) = \inf \set{ \mu^+(B) \;;\; \mu(B) = v } ( \; = \I_{(\Omega,d,\mu)}(v) \; ) ~.
\]
In general, isoperimetric minimizers of given measure need not necessarily exist; however, that is not the case in our setup.

Indeed, given a complete smooth oriented $n$-dimensional Riemannian manifold $(M,g)$, a domain (open connected set) $\Omega \subset M$, and a positive density $\Psi$ on $\overline{\Omega}$ so that $\log \Psi \in C^1(\overline{\Omega})$,
define the \emph{$\Psi$-weighted volume} of a Borel set $A \subset \Omega$ as:
\[
V_{\Psi}(A) := \int_A \Psi(x) dvol_g(x) ~,
\]
and the \emph{$\Psi$-weighted relative perimeter in $\Omega$} as:
\[
P_{\Psi}(A,\Omega) := \sup \set{ \int_A \brac{div(X) + g(\nabla \log \Psi,X)} \Psi(x) dvol_g(x) ; g(X,X) \leq 1 } ~,
\]
where $X$ is a $C^1$ smooth vector field over $M$ with compact support contained in $\Omega$, and $div(X)$ denotes the divergence of $X$.
When $\Psi \equiv 1$, we will simply write $V(A)$ and $P(A,\Omega)$.
It follows immediately by the Gauss--Green Divergence Theorem, that when $\partial A \cap \Omega$ is nice enough (say $C^2$), then:
\[
P_{\Psi}(A,\Omega) = \int_{\partial A \cap \Omega} \Psi(x) dvol_{\partial A}(x) ~.
\]
More generally, as in the constant density case (see \cite{GiustiBook,BuragoZalgallerBook}), it follows from the Gauss--Green--De Giorgi--Federer theorem (\cite[4.5.6]{FedererBook}, \cite[Chapter 12]{MorganBook4Ed}) that:
\begin{equation} \label{eq:gen-reduced-bdry}
P_{\Psi}(A,\Omega) = \int_{\partial^* A \cap \Omega} \Psi(x) d\H^{n-1}(x) ~,
\end{equation}
where $\partial^* A$ is the reduced boundary of $A$ and $\H^{k}$ denotes $k$-dimensional Hausdorff measure.
We refer to \cite{GiustiBook,BuragoZalgallerBook} for the definition of reduced boundary, and only remark that it is defined as the set of points where a unique measure-theoretic normal exists; in particular, it contains any point $x \in \partial A$ for which $\partial A$ is (say) $C^1$ smooth in a neighborhood of $x$. In addition, as in the case of constant density, it easily follows (e.g. \cite{GiustiBook,BayleRosales}) that $P_{\Psi}(A,\Omega)$ is lower-semi-continuous with respect to convergence of sets in the $L^1_{loc}(\Psi) := L^1_{loc}(\Omega,\Psi vol_g)$ topology, i.e.:
\[
A_i \rightarrow_{L^1_{loc}(\Psi)} A \;\;\; \Rightarrow \;\;\; P_{\Psi}(A,\Omega) \leq \liminf_{i \rightarrow \infty} P_{\Psi}(A_i,\Omega) ~,
\]
where $A_i \rightarrow_{L^1_{loc}(\Psi)} A$ if for any compact $K \subset \Omega$:
\[
\lim_{i \rightarrow \infty} \int_{K} |1_{A_i}(x) - 1_{A}(x)| \Psi(x) dvol_g(x) = 0 ~.
\]

Assuming in addition that $V := \int_{\Omega} \Psi(x) dvol_g(x) < \infty$, it is well known (see e.g. \cite[Sections 5.5,9.1]{MorganBook4Ed}, \cite{MorganManifoldsWithDensity}) that there exists a $\Psi$-weighted relative perimeter minimizer in $\Omega$ of any $\Psi$-weighted volume $v \in (0,V)$ (the cases $v=0,V$ are obvious). Indeed, let $\set{A_i}$ denote a sequence of subsets minimizing perimeter of a given volume $v$, i.e.  $V_{\Psi}(A_i) = v$ and $\lim_{i \rightarrow \infty} P_{\Psi}(A,\Omega) = \inf \set{P_{\Psi}(A_i,\Omega) ; V_{\Psi}(A) = v }$. By restricting to a sequence of increasing balls exhausting $M$, employing on each ball the local compactness theorem for $BV$ functions (e.g. \cite[Theorem 1.19]{GiustiBook} whose proof carries over to our setup),
passing to a convergent subsequence in $L^1_{loc}(\Psi)$ and employing a standard diagonalization argument, it follows that there exists a set $A \subset \Omega$ so that $A_i$ converges to $A$ in $L^1_{loc}(\Psi)$ (see e.g. the proof of \cite[Theorem 2.1]{RitoreRosalesMinimizersInCones} for more details).
Using that $V < \infty$, it follows that the convergence is in fact in $L^1(\Psi)$ globally, hence $V_{\Psi}(A) = v$, and the lower-semi-continuity of perimeter concludes the standard claim.

\medskip

Now assume $V=1$ and denote $\mu = \Psi \cdot vol_g$. For a general Borel set $A \subset \Omega$, it is known that $\mu^+(A) \geq P_{\Psi}(A,\Omega)$ (see the proofs of \cite[Theorem 14.2.1]{BuragoZalgallerBook} or \cite[Theorem III.4.1]{Chavel-IsopBook} which carry over to our setup),
but in general the reverse inequality may be false. However, for a set $A$ minimizing $\Psi$-weighted perimeter in $\Omega$, we will see below that equality does hold, and so $A$ must also minimize Minkowski's notion $\mu^+$ of boundary measure, yielding the existence of isoperimetric minimizers of any given measure $v \in (0,1)$. We conclude that ultimately it does not matter with which definition of boundary measure one works with, and our choice of using Minkowski's exterior boundary measure is only a matter of expositional convenience (along with some convenient approximation properties as in Section \ref{sec:gen}).

\medskip

We now turn to describe the known regularity results for the boundary of an isoperimetric minimizer $\overline{\partial A \cap \Omega}$. In the Euclidean setting, the following regularity results are a consequence of \cite{DeGiorgiFirstRegularity,FedererRegularity,FedererBook,GMT,Gruter} (see also \cite{GiustiBook,MorganBook4Ed}); the extension to the Riemannian setting follows easily from the methods of Federer \cite[5.3.19]{FedererBook} (cf. \cite[8.5]{MorganBook4Ed} and \cite{TamaniniExtendedExcessFuncAndHolderRegularity}) or Almgren \cite{AlmgrenMemoirs} (whose approach was elucidated by Bombieri \cite{BombieriRegularityTheory} and Morgan \cite{MorganRegularityOfMinimizers}); the extension to the manifold-with-density setting is due to Morgan \cite[Remark 3.10]{MorganRegularityOfMinimizers} (see also \cite{BayleThesis} and \cite[Chapter 18]{MorganBook4Ed}).
All of the results we require are summarized in the following:

\begin{thm}[Almgren, De Giorgi, Federer, Giusti, Gonzales--Massari--Tamanini, Gr\"{u}ter, Morgan] \label{thm:regularity}
Let $(M,g)$ denote a complete smooth oriented $n$-dimensional Riemannian manifold, let $\Omega \subset M$ denote a domain with (possibly empty) $C^2$ boundary, and let $\mu$ denote a probability measure supported on $\overline{\Omega}$ so that $\mu = \Psi \cdot vol_g|_\Omega$, with $\Psi > 0$ on $\overline{\Omega}$ and $\log \Psi \in C^k(\overline{\Omega})$, for some $k \geq 2$.
Then for any $v\in(0,1)$, there exists a Borel $\Psi$-weighted relative perimeter minimizer $A \subset \Omega$ of $\Psi$-weighted volume $v$.
The interior boundary $\partial_i A := \overline{\partial A \cap \Omega}$ can be written as a disjoint union of a relatively open
regular part $\partial_r A$ and a closed set of singularities $\partial_s A$, with the following properties:
\begin{itemize}
\item
$\partial_r A \cap \Omega$ is a $C^k$-smooth, embedded hypersurface with outward unit normal $\nu(x)$ and constant generalized mean-curvature:
\[
\forall x \in \partial_r A \cap \Omega \;\;\; H^\nu_{\partial_r A \cap \Omega,\mu}(x) =: H_\mu(A) ~.
\]
\item
If $x \in \partial_r A \cap \partial \Omega$, then in a neighborhood of $x$, $\partial_r A$ is a $C^1$-smooth,
embedded hypersurface with boundary contained in $\partial \Omega$. In this neighborhood $\partial_r A$
meets $\partial \Omega$ orthogonally.
\item
$\partial_s A$ is a closed set of Hausdorff dimension at most n-8.
\end{itemize}
\end{thm}

\begin{rem}
It is known that the constant 8 above is sharp \cite{BombieriDeGiorgiGiusti,Gruter-OptimalRegularityForCodimensionOneMinimalSurfacesWithAFreeBoundary}.
Note that changing a set by zero $\H^n$-measure does not change its perimeter, and so the regularity results above ensure that we may replace $A$ by the open set $A \setminus \partial_i A$ without changing its $\Psi$-volume and $\Psi$-perimeter; we will subsequently always assume that our minimizer is an open set.
The fact that the generalized mean-curvature is constant on $\partial_r A \cap \Omega$ follows immediately by a Lagrange multiplier argument and Remark \ref{rem:first-variation}, since otherwise we could deform $\partial_r A \cap \Omega$ so that in the first order $\Psi$-weighted volume is preserved whereas $\Psi$-weighted perimeter is decreased, contradicting the minimality of $A$ (see e.g. \cite[Proposition 3.4.11]{BayleThesis} or \cite[Section 2]{EMilmanGeometricApproachPartI}).
\end{rem}

The following consequence is elementary (cf. \cite[pp. 32--33 and Appendix A]{BayleThesis}):
\begin{cor} \label{cor:boundary-measure}
With the same assumptions and notation as in Theorem \ref{thm:regularity}, we have:
\[
\mu^+(A) = P_{\Psi}(A,\Omega) = \int_{\partial_r A \cap \Omega} \Psi(x) dvol_{\partial_r A}(x) ~.
\]
Consequently, $A$ is an isoperimetric minimizer of $\mu$-measure $v$.
\end{cor}
\begin{proof}[Proof sketch]
We will sketch the proof when $\overline{\Omega}$ is compact, otherwise we may exhaust $\Omega$ by compact sets and use the fact the total measure is finite. It follows in the compact case that when $\eps>0$ is small enough:
\[
\frac{\mu( (\partial_r A)_\eps \setminus A ) }{\eps} \leq \frac{\mu(A_\eps) - \mu(A)}{\eps} = \frac{\mu( (\partial_i A)_\eps \setminus A ) }{\eps} \leq \frac{\mu( (\partial_r A)_\eps \setminus A ) }{\eps} + \frac{\mu( (\partial_s A)_\eps \setminus A ) }{\eps} ~.
\]
For the regular part of the boundary, it is elementary (see e.g. \cite[Remark III.2.3]{Chavel-IsopBook} or \cite[pp. 32--33]{BayleThesis}) to verify that:
\[
\lim_{\eps \rightarrow 0} \frac{\mu( (\partial_r A)_\eps \setminus A ) }{\eps} = \lim_{\eps \rightarrow 0} \frac{\mu( (\partial_r A \cap \Omega)_\eps \setminus A ) }{\eps} = \int_{\partial_r A \cap \Omega} \Psi(x) dvol_{\partial_r A}(x) \leq P_\Psi(A,\Omega) ~,
\]
where the last inequality follows from (\ref{eq:gen-reduced-bdry}).
Since always $\mu^+(A) \geq P_{\Psi}(A,\Omega)$, we will obtain an equality if we show that:
\[
\lim_{\eps \rightarrow 0} \frac{\mu( (\partial_s A)_\eps \setminus A ) }{\eps} \leq \lim_{\eps \rightarrow 0} \frac{\mu( (\partial_s A)_\eps \setminus \partial_s A ) }{\eps} =  0
\]
(note that the last inequality follows since $\mu(\partial_i A) = 0$ thanks to the regularity of $A$). But the latter requirement is an immediate consequence of the low Hausdorff dimension of $\partial_s A$ and the boundedness of $\Psi$ and the geometry of $M$ on compact sets. This concludes the proof.
\end{proof}

\subsection{Oriented Tangent Cones}

Furthermore, we will require the following information on the existence and minimization properties of oriented tangent cones to $A$.
In addition to the already mentioned references above, we refer to \cite{GruterBoundaryOfArbitraryCoDimension,GruterJostRegularityofAllardTypeOnFreeBoundary} for further information on properties of oriented tangent cones at the boundary of $\Omega$. Here and elsewhere, we equip $T_x M$ with its natural Euclidean metric $g_x$, and denote by $B^n$ and $S^{n-1}$ the open unit ball and sphere in $T_x M$, respectively.

\begin{dfn*}
We say that $C^+ \subset T_x M$ is an \emph{oriented tangent cone} to an $n$-dimensional Borel set $\Sigma \subset M$ at $x \in M$, if the following conditions are satisfied:
\begin{itemize}
\item $C^+ \subset T_x M$ is an oriented cone, i.e. $t C^+ = C^+$ for all $t > 0$.
\item Given $\eps > 0$ smaller than the injectivity radius at $x$, define $B := \exp_x^{-1}(B_x(\eps) \cap \Sigma) \subset T_x M$, where $B_x(\eps)$ is a geodesic ball of radius $\eps$ centered at $x$. There exists a sequence $\set{t_k}$ of positive reals tending to $0$, so that $1_{B/t_k}$ tends to $1_{C^+}$ in $L^1_{loc}(T_x M)$.
\end{itemize}
\end{dfn*}

\begin{dfn*}
For a point $x \in \overline{\Omega}$, we denote by $T_x \Omega \subset T_x M$ the open oriented tangent cone to $\Omega$. In other words:
\[
T_x \Omega := \begin{cases} T_x M & x \in \Omega \\  \set{v \in T_x M ; g_x(v,n_x) > 0} & x \in \partial \Omega \end{cases} ~,
\]
where $n_x$ denotes the inward pointing unit normal vector to $\partial \Omega$ at $x$.
\end{dfn*}

\begin{dfn*}
Given an oriented cone $C^+ \subset T_x M$, we define its boundary relative to $\Omega$ as $\partial_{\Omega} C^+ := \overline{\partial C^+ \cap T_x \Omega}$. The latter's density is defined as:
\[
\Theta(\partial_{\Omega} C^+) := \frac{\H^{n-1}(\partial_{\Omega} C^+ \cap B^n)}{\H^{n-1}(E \cap T_x \Omega \cap B^{n})} ~,
\]
where $E \subset T_x M$ is any hyperplane through the origin orthogonal to $\partial T_x \Omega$.
\end{dfn*}

\begin{thm}[De Giorgi, Federer, Giusti, Gonzales--Massari--Tamanini, Gr\"{u}ter, Morgan]
With the same assumptions and notation as in Theorem \ref{thm:regularity}:
\begin{itemize}
\item
At every point $x \in \partial_i A$, $A$ has a closed oriented tangent cone $C_x^+ \subset T_x M$.
\item
$\partial_\Omega C_x^+$ is an $n-1$ dimensional rectifiable set with $1 \leq \Theta(\partial_\Omega C_x^+) < \infty$.
\item
$C_x^+$ is perimeter-minimizing in $T_x \Omega$:
for any bounded open $\Sigma \subset T_x \Omega$ and Borel competitor $E \subset T_x \Omega$ so that $\overline{C_x^+ \triangle E} \cap T_x \Omega \subset \Sigma$, $P(E,\Sigma) \geq P(C_x^+,\Sigma)$.
\end{itemize}
\end{thm}

\begin{rem} \label{rem:hyperplane}
The regular part $\partial_r A$ of the boundary $\partial_i A$ is precisely characterized as the collection of those points $x \in \partial_i A$ for which $\partial_\Omega C_x$ is a hyperplane ($x \in \Omega$) or half hyperplane ($x \in \partial \Omega$).
\end{rem}

\begin{rem}
In some of the above mentioned references, the results above are demonstrated for the oriented tangent cone of the (locally) integral current associated to $A$, which is by itself a (locally) integral current. However, it is known (e.g. \cite[p. 110]{MorganBook4Ed}) that the support of the boundary of such a current and the reduced boundary of its support coincide up to zero $\H^{n-1}$-measure, so ultimately it does not matter which definition of tangent cone one uses.
\end{rem}

\subsection{Any closest point on a minimizer's boundary is regular}

We are now ready to state the following crucial proposition which extends a fundamental observation of Gromov \cite{GromovGeneralizationOfLevy}, who addressed the case when the density is constant and $\Omega$ has no boundary. Recall that a set $\Omega \subset (M,g)$ is called geodesically convex if between any two points $x,y \in \Omega$ there exists a distance minimizing geodesic (not necessarily unique) which lies entirely in $\Omega$.

\begin{prop} \label{prop:closest-point}
Assume in addition to the assumptions of Theorem \ref{thm:regularity} that $\Omega$ is geodesically convex. Proceeding with the same notation as in that theorem, let $A$ denote an isoperimetric minimizer. Then for any $p \in \Omega \setminus \partial_i A$, any closest point to $p$ in $\partial_i A$ must lie in $\partial_r A \cap \Omega$.
\end{prop}

In other words, the claim is that the outward pointing normal rays (with respect to $A$) emanating from $\partial_r A \cap \Omega$ sweep out the entire $\Omega \setminus \overline{A}$, and similarly the open $A$ is swept out by the inward pointing normal rays. For the proof, we require the following lemma, which is an easy corollary of the tangent cones' minimizing properties. Since we could not find a reference (at least for the second part), we include a proof for completeness:
\begin{lem} \label{lem:gravity}
With the same assumptions and notation as in Theorem \ref{thm:regularity}, let $C^+_x$ denote a closed oriented tangent cone to $A$ at $x \in \partial_i A$, and let $p_x$ denote the center of gravity of $\partial_\Omega C^+_x$, defined as:
\[
p_x := \int_{\partial_\Omega C^+_x \cap B^n} \theta \; d\H^{n-1}(\theta) ~.
\]
\begin{itemize}
\item If $x \in \Omega$ then $p_x = 0$.
\item If $x \in \partial \Omega$ then $p_x$ is a positive multiple of $n_x$, the inward unit normal vector to $\partial \Omega$ at $x$.
\end{itemize}
\end{lem}
\begin{proof}
Denote for short $C^+ = C_x^+$ and $C = \partial_{\Omega} C_x^+$, and set $B^+ := C^+ \cap S^{n-1}$ and $B := C \cap S^{n-1}$. It is known that $B^+$ and $B$ are $n-1$ and $n-2$ dimensional rectifiable closed sets, respectively.
Denote by $B^* \subset B$ the part in $B$ of the reduced boundary of $B^+$ in $S^{n-1}$, having a well-defined outward normal to $B^+$ at every $\theta \in B^*$, which we denote by $n_{S^{n-1},B}(\theta) \in T_{\theta} S^{n-1}$. Denote by $K_v^\infty$ the oriented cone over $B$ with vertex $v \in B^n$, and set $K_v = K_v^\infty \cap \overline{B^n}$ i.e. $K_v = \cup_{b \in B} [v,b]$. Similarly, denote by $K_v^+$ the oriented cone in $\overline{B^n}$ over $B^+$ with vertex $v$.
We claim that for $v \in B^n$:
\begin{equation} \label{eq:cone-formula}
\H^{n-1}(K_v) = \frac{1}{n-1} \int_{B^*} \sqrt{ (1-\scalar{v,\theta})^2 + \scalar{v,n_{S^{n-1},B}(\theta)}^2} d\H^{n-2}(\theta) ~.
\end{equation}
Indeed, by the Gauss--Green--De Giorgi--Federer theorem (\cite[4.5.6]{FedererBook}, \cite[Chapter 12]{MorganBook4Ed}), we have:
\[
\H^{n-1}(K_v) = \frac{1}{n-1} \int_{K_v} div(w-v) d\H^{n-1}(w) = \frac{1}{n-1} \int_{B_v^*} \scalar{\theta-v,n_{K_v,B}(\theta)} d\H^{n-2}(\theta) ~,
\]
where $B_v^*$ denotes the reduced boundary of $K_v$ inside $K_v^\infty$,
and $n_{K_v,B}(\theta)$ denotes the outward normal to $K_v$ at $\theta \in B^*_v$ in the cone $K_v^\infty$. It is easy to see that $B_v^* = B^*$ 
and that given $\theta \in B^*$ and denoting $E = span\set{n_{S^{n-1},B}(\theta),\theta}$, that $n_{K_v,B}(\theta) = P_{E}(\theta-v) / |P_E(\theta-v)|$ (where $P_E$ denotes orthogonal projection onto $E$). Therefore  $\abs{\scalar{n_{K_v,B}(\theta),\theta-v}} = \abs{P_E(\theta-v)}$, and as $P_E(\theta-v) = \theta - \scalar{v,\theta} \theta - \scalar{v,n_{S^{n-1},B}(\theta)} n_{S^{n-1},B}(\theta)$, (\ref{eq:cone-formula}) follows.

Differentiating (\ref{eq:cone-formula}) in $v$, one verifies that the first variation of $v \mapsto \H^{n-1}(K_v)$ in the direction of $\xi \in T_x M$ at $v=0$ (we identify $T_0 T_x M$ with $T_x M$) is:
\[
- \frac{1}{n-1} \int_{B^*} \scalar{\theta,\xi} d\H^{n-2}(\theta) = - \frac{1}{n-1} \int_{B} \scalar{\theta,\xi} d\H^{n-2}(\theta)
\]
(the last equality follows since $\H^{n-2}(B \setminus B^*) = 0$ by the classical regularity results for area-minimizing sets of finite perimeter).

When $x \in \Omega$, note that $P(K_v^+,B^n) = \H^{n-1}(K_v)$ for all $v \in B^n$ and that $K_0^+ = C^+$. Since $C^+$ is perimeter minimizing among all competitors in $2 B^n$, it follows that the first variation must be $0$ for all $\xi \in T_x M$, and so the center of gravity must be at the origin.

When $x \in \partial \Omega$, note that $P(K_v^+,B^n \cap T_x \Omega) = \H^{n-1}(K_v)$ for all $v \in B^n \cap \partial T_x \Omega$ and that $K_0^+ = C^+$. It follows as above that the first variation must be $0$ for all $\xi \in \partial T_x \Omega$, and so the center of gravity must lie on the ray spanned by $n_x$. Moreover, since the density of $C$ is positive, this must be a strictly positive multiple. This concludes the proof.

\end{proof}

\begin{proof}[Proof of Proposition \ref{prop:closest-point}]
Given a minimizer $A$ and a point $p \in \Omega \setminus \partial_i A$, let $x \in \partial_i A$ denote a closest point to $p$. We first claim that $x \notin \partial \Omega$. Arguing in the contrapositive, let $n_x \in T_x M$ denote the inward normal to $\partial \Omega$ at $x \in \partial \Omega$. Since $\Omega$ is geodesically convex then so is $\overline{\Omega}$ (this easily follows from the Arzel\`a-Ascoli theorem, see e.g. the proof of \cite[Proposition 2.5.3]{Papadopoulos-MetricSpacesConvexityAndCurvatureBook}). It follows that there exists a distance minimizing geodesic $[0,a] \in t \mapsto \gamma(t) = \exp_x(t v)$ between $x$ and $p$ which lies inside $\overline{\Omega}$ (e.g. for some $v \in T_x M$ with $|v|=1$). Clearly $g_x(v,n_x) \geq 0$, since otherwise $\gamma(t) \notin \overline{\Omega}$ for $t \in (0,\eps)$ for some small $\eps > 0$. If $g_x(v,n_x) = 0$ this means that the geodesic is tangent to $\partial \Omega$, and since they are both closed sets, defining $t_0 = \inf \set{t > 0 ; \gamma(t) \notin \partial \Omega }$, it follows that $x_0 = \gamma(t_0) \in \partial \Omega$ and that $\gamma$ is still tangent to $\partial \Omega$ at $x_0$. By a result of Bishop \cite{BishopInBayleRosales}, the geodesic convexity of a domain with $C^2$ smooth boundary implies that geodesics tangent to the boundary must be locally outside the domain, and so for some small enough $\eps > 0$, $\set{\gamma(t)}_{t \in (t_0,t_0+\eps]}$ must be outside $\Omega$, outside $\partial \Omega$, and inside $\overline{\Omega}$, a contradiction. It remains to exclude the case that $g_x(v,n_x) > 0$. Note that $C^+_x \subset \set{ y \in T_x M ; g_x(y,v) \leq 0 }$, where $C^+_x$ denotes a closed oriented tangent cone to $A$ at $x$, since otherwise we could shorten the distance from $p$ to $A$ by the first variation of distance formula. Consequently, if $g_x(v,n_x) > 0$, it would be impossible for the center of gravity of $\partial_\Omega C_x^+$ to be a strictly positive multiple of $n_x$, obtaining a contradiction to Lemma \ref{lem:gravity}.

We have shown that $x \in \partial A \cap \Omega$, and it remains to show that $x \in \partial_r A \cap \Omega$. Employing the same notation as before, it still holds that $C_x^+ \subset \set{ y \in T_x M ; g_x(y,v) \leq 0 }$. Since in addition we know by Lemma \ref{lem:gravity} that the center of gravity of $\partial C_x$ must be at the origin, it necessarily follows that $\partial C_x = \set{ y \in T_x M ; g_x(y,v) = 0 }$. But by Remark \ref{rem:hyperplane}, this precisely characterizes regular boundary points, and the assertion now follows.
\end{proof}

\subsection{Main Tool}

Combining all the information contained in this section, we derive our main tool in this work:

\begin{thm} \label{thm:main-tool}
Let $(M^n,g,\mu)$ satisfy the $CDD(\rho,n+q,D)$ condition with $\rho \in \Real$, $q \in [0,\infty]$ and $D \in (0,+\infty]$.
Given $v \in (0,1)$, let $A$ denote an open isoperimetric minimizer with $\mu(A) = v$. Denote by $H_\mu(A)$ the constant generalized curvature of $\partial_r A \cap \Omega$. Then there exist $r_A,r_{A^c} > 0$ with $r_A + r_{A^c} \leq D$ so that:
\begin{eqnarray*}
1-v = \mu(\Omega \setminus A) & \leq & \mu^+(A) \int_0^{r_{A^c}} J_{H_\mu(A),\rho,n+q-1}(t) dt ~, \\
v = \mu(A) &\leq & \mu^+(A) \int_0^{r_A} J_{-H_\mu(A),\rho,n+q-1}(t) dt ~.
\end{eqnarray*}
\end{thm}

\begin{proof}
Denote $A^c := \Omega \setminus \overline{A}$, and set:
\[
r_{A} := \sup \set{ d(x,\partial A \cap \Omega) ; x \in A} ~,~ r_{A^c} := \sup \set{ d(y,\partial A \cap \Omega) ; y \in A^c} ~.
\]
Since between any $x \in A$ and $y \in A^c$ there exists a distance minimizing geodesic contained in $\Omega$ (by convexity of $\Omega$), it must intersect $\partial A \cap \Omega$. Consequently, we obviously have $r_A + r_{A^c} \leq D$, where recall $D$ is an upper bound on the diameter of $\Omega$.

Applying Proposition \ref{prop:closest-point}, we are ensured that the outward pointing normal rays (with respect to $A$) emanating from $\partial_r A \cap \Omega$ and extending to a distance of $r_{A^c}$ will sweep out the entire $A^c$, and that similarly, the inward pointing normal rays extending to a distance of $r_A$ will sweep out the entire $A$. Applying the Generalized Heintze--Karcher Theorem \ref{thm:gen-HK} to the hypersurface $\partial_r A \cap \Omega$, and noting that the resulting Jacobian term is constant along it, the assertions immediately follows (taking into account Corollary \ref{cor:boundary-measure}).

Before concluding, we remark that as usual, $\mu(A^c) = \mu(\Omega \setminus A)$ thanks to the regularity of $\partial_i A$, and that the difference between using outward and inward normal rays amounts to changing the sign of $H_\mu(A)$ in the Jacobian term.
\end{proof}

\section{An Isoperimetric Inequality} \label{sec:proof}

In this section, we provide a proof of Theorem \ref{thm:main1}, first for $n \geq 2$, and subsequently for the elementary case $n=1$.

\begin{proof}[Proof of Theorem \ref{thm:main1}]
We may clearly assume that $D<\infty$ or that $\rho>0$, otherwise there is nothing to prove.

Let $v \in (0,1)$, and let $A \subset \Omega$ denote an open isoperimetric minimizer for $(M,g,\mu)$ with $\mu(A) = v$. Theorem \ref{thm:main-tool} states that there exist $r_A,r_{A^c} > 0$ with $r_A + r_{A^c} \leq D$ so that:
\[
\mu^+(A) \geq \max\brac{\frac{v}{\int_{0}^{r_A} J_{-H_\mu(A)}(t) dt}, \frac{1-v}{\int_0^{r_{A^c}} J_{H_\mu(A)}(t) dt}} ~,
\]
where we denote for brevity $J_H := J_{H,\rho,n+q-1}$. Noting that $J_{-H}(t) = J_{H}(-t)$, it follows in particular that:
\begin{equation} \label{eq:proof1}
\I(v) \geq \inf_{H \in \Real, a \in [D-D,D]} \max\brac{\frac{v}{\int_{-a}^{0} J_{H}(t) dt},\frac{1-v}{\int_0^{D-a} J_H(t) dt}} ~.
\end{equation}
Note that the first (second) term inside the maximum on the right-hand side above is continuously monotone non-increasing from $\infty$ (non-decreasing to $\infty$) in $a \in [0,D]$.
 Consequently, if $D<\infty$, then given any $H \in \Real$, the infimum over $a \in [0,D]$ of this maximum is attained at a point $a_{H,v}$ when both terms are equal, i.e. precisely when:
\begin{equation} \label{eq:proof2}
\frac{\int_{-a_{H,v}}^0 J_H(t) dt}{v} = \frac{\int_{0}^{D-a_{H,v}} J_H(t) dt}{1-v} = \int_{-a_{H,v}}^{D-a_{H,v}} J_H(t) dt ~.
\end{equation}
Denoting by $\mu_{H,v}$ the probability measure on $[-a_{H,v},D-a_{H,v}]$ having density proportional to $J_H(t)$, note that $\mu_{H,v}([-a_{H,v},0]) = v$ and that $\mu_{H,v}^+([-a_{H,v},0]) = 1 / \int_{-a_{H,v}}^{D-a_{H,v}} J_H(t) dt$ as $J_H(0) = 1$. Consequently, we deduce from (\ref{eq:proof1}) and (\ref{eq:proof2}) that:
\[
\I(v) \geq \inf_{H \in \Real} \mu_{H,v}^+([-a_{H,v},0]) \geq \inf_{H \in \Real} \J(J_H(t),[-a_{H,v},D-a_{H,v}])(v) ~,
\]
and in particular:
\[
\I(v) \geq \inf_{H \in \Real, a \in [0,D]} \J(J_H(t),[-a,D-a])(v) ~.
\]
When $D=\infty$ (so necessarily $\rho > 0$), one verifies that the second (first) term inside the maximum in (\ref{eq:proof1}) varies monotonically from $0$ to $\infty$ ($\infty$ to $0$) as $H$ varies from $-\infty$ to $\infty$, and so the infimum over $H \in \Real$ of this maximum is again attained at the unique point $H_v$ when both terms are equal:
\[
\frac{\int_{-\infty}^0 J_{H_v}(t) dt}{v} = \frac{\int_{0}^{\infty} J_{H_v}(t) dt}{1-v} = \int_{-\infty}^{\infty} J_{H_v}(t) dt
\]
(note that these expressions are finite since $\rho > 0$). The rest of the argument is identical to the one already described above, thereby concluding the proof.

Note that fixing $a$ (such that $a>0$ and $D-a > 0$) and varying $H$ also does the job when $D < \infty$, but we preferred to fix $H$ and vary $a$ in this case, as this may be more intuitive.
\end{proof}

\begin{rem}
It follows from the proof of Theorem \ref{thm:main1} that given $v \in (0,1)$, we need only to take an infimum over $H$ \emph{or} $a$ in (\ref{eq:main1}), and that the other parameter ($a$ or $H$) is actually determined by the former one together with $v$, so it seems as though we are being wasteful here. However, as we shall see in Corollary \ref{cor:inf2=inf1} below, we actually do not lose here at all (if we did, our bounds could not be sharp, as claimed in Theorem \ref{thm:main2}). The infimum over the second parameter serves both an aesthetic purpose, as well as enabling us to identify more easily the different model spaces in the proof of Corollary \ref{cor:main1}, where thanks to algebraic properties of the function $J_H(t)$, the infimum over both parameters simplifies to an infimum over a single equivalent one.
\end{rem}

\begin{cor} \label{cor:main1-1D}
Theorem \ref{thm:main1} also holds in the one-dimensional case, i.e. for any space $(\Real,|\cdot|,\mu)$ which satisfies the $CDD(\rho,1+q,D)$ condition.
\end{cor}
\begin{proof}
Given $v \in (0,1)$, it still holds that there exists an open minimizer $A \subset \Real$ with $\mu(A) = v$ and $\mu^+(A) = \I(v) < \infty$. We may clearly assume without loss of generality that $\Omega \setminus A$ does not have isolated points, since those will not influence $\mu(A)$ nor $\mu^+(A)$. In that case, denoting $\partial_r A := \partial A \cap \Omega$, it follows easily that
$\mu^+(A) = \sum_{x \in \partial_r A} \Psi(x)$, and that for all $x \in \partial_r A$, $\nu(x) (\log \Psi)'(x) = H_\mu(A)$ is constant, where $\nu(x)$ is equal to $+1$ ($-1$) if $x$ is a right (left) boundary point. The point here is that Theorem \ref{thm:main-tool} remains valid in the following form:
\begin{equation} \label{eq:1D-1}
\mu(A_r) - \mu(A) \leq \mu^+(A) \int_0^r J_{H_\mu(A),\rho,q}(t) dt ~;
\end{equation}
in fact, the proof of Theorem \ref{thm:gen-HK} ultimately reduces to this one-dimensional case. The latter follows from a well-known elementary argument, which we reproduce here for completeness when $q \in (0,\infty)$; the case $q = \infty$ follows similarly and the case $q=0$ is obvious. Indeed, a simple application of the maximum principle ensures that since $\Psi^{1/q} \in C^2(\Omega)$ satisfies by assumption:
\[
 \frac{d^2}{dx^2}\Psi^{1/q} + \frac{\rho}{q} \Psi^{1/q} \leq 0 \text{ on } \Omega ~,
\]
then for any $x \in \Omega$ and $t \in \Real$ so that $[x,x+t] \subset \Omega$ we have (see e.g. \cite[Theorem 4.19]{GHLBookEdition3} or \cite{GromovGeneralizationOfLevy}):
\begin{equation} \label{eq:1D-2}
\frac{\Psi^{1/q}(x+t)}{\Psi^{1/q}(x)} \leq h_+(t) ~,
\end{equation}
where $h(t)$ denotes the solution to:
\[
\frac{d^2}{dt^2} h + \frac{\rho}{q} h = 0  ~,~ h(0) = 1 ~,~ h'(0) = \frac{(\Psi^{1/q})'(x)}{\Psi^{1/q}(x)} = \frac{1}{q} (\log \Psi)'(x) ~.
\]
Since:
\[
h(t) = c_\delta(t) + \frac{(\log \Psi)'(x)}{q} s_\delta(t) ~,~ \delta = \frac{\rho}{q} ~,
\]
we conclude together with (\ref{eq:1D-2}) that in the above range:
\[
\Psi(x+t) \leq \Psi(x) J_{(\log \Psi)'(x),\rho,q}(t) ~.
\]
Combining all of the above, (\ref{eq:1D-1}) immediately follows. The rest of the proof of Theorem \ref{thm:main1} remains unchanged.
\end{proof}

\begin{cor} \label{cor:inf2=inf1} \label{cor:replace-J-with-I}
For any $\rho \in \Real$, $m \in [0,\infty]$, $D \in (0,+\infty]$ and $v \in [0,1]$, we have:
\begin{eqnarray}
\label{eq:simplified1} &   & \inf_{H \in \Real, a \in [D-D,D]} \I(J_{H,\rho,m},[-a,D-a])(v) \\
\label{eq:simplified2} & = & \inf_{H \in \Real, a \in [D-D,D]} \J(J_{H,\rho,m},[-a,D-a])(v) \\
\label{eq:simplified3} & = & \;\;\;\;\;\;\;\; \inf_{H \in \Real} \J(J_{H,\rho,m},[-a_H,D-a_H])(v) ~.
\end{eqnarray}
Here $a_H \in [D-D,D]$ is chosen when $D < \infty$ so that $\mu_{H,v}((-\infty,0]) = v$, where $\mu_{H,v}$ denotes the probability measure supported in $[-a_H,D-a_H]$ with density proportional to $J_{H,\rho,m}$ there, i.e.:
\begin{equation} \label{eq:a_H}
\frac{v}{1-v} = \frac{\int_{-a_H}^0 J_{H,\rho,m}(t) dt}{\int_{0}^{D-a_H} J_{H,\rho,m}(t) dt} ~;
\end{equation}
when $D = \infty$, we set $a_H = \infty$.
\end{cor}
\begin{proof}
Note that all of the above expressions are $0$ if $D=\infty$ and $\rho \leq 0$, and that obviously the assertion holds with the ``$=$"'s in (\ref{eq:simplified2}) and (\ref{eq:simplified3}) replaced by ``$\leq$"'s.
Observe that (\ref{eq:simplified3}) is precisely the lower bound ensured by the proof of Theorem \ref{thm:main1} when $m=n+q-1$; in particular, Corollary \ref{cor:main1-1D} implies that this is a lower bound on the boundary measure of sets having $v$ measure in any one-dimensional space satisfying the $CDD(\rho,m+1,D)$ condition. Applying this lower bound to all one-dimensional spaces $(\Real,\abs{\cdot},\mu_{H,a})$ for $H \in \Real, a \in [D-D,D]$, where $\mu_{H,a}$ is the probability measure having density proportional to $J_{H,\rho,m}$ on the interval $[-a,D-a]$ (note that if $D<\infty$ or $\rho > 0$ this is always possible), it follows that (\ref{eq:simplified1}) must also be bounded below by (\ref{eq:simplified3}), concluding the proof of the equivalence.

Note that when $J_{H,\rho,m}$ is log-concave, as mentioned in the Introduction, \cite[Proposition 2.1]{BobkovExtremalHalfSpaces} implies that $\I(J_{H,\rho,m},[-a,D-a]) = \J(J_{H,\rho,m},[-a,D-a])$. This is indeed the case for all $H \in \Real$ when $\rho \geq 0$, but may fail to be true when $\rho < 0$.
\end{proof}

\section{Families of Model Spaces} \label{sec:model}

In this section we provide a proof of Corollary \ref{cor:main1}, which identifies the various one parameter families of model spaces for the $CDD(\rho,n+q,D)$ condition, for different values of $\rho,q,D$. Note that by definition $CDD(\rho_1,n+q_1,D_1) \Rightarrow CDD(\rho_2,n+q_2,D_2)$ if $\rho_2 \leq \rho_1$, $q_2 \geq q_1$ and $D_2 \geq D_1$.

\medskip
We shall first require the following simple:
\begin{lem} \label{lem:profile-decreasing}
Let $f: \Real \rightarrow \Real_+$ denote a log-concave function (meaning that $- \log f: \Real \rightarrow \Real \cup \set{+\infty}$ is convex). Then given $v \in (0,1)$, the function $(a,b) \mapsto \J(f,[a,b])(v)$ is non-increasing in $b$ and non-decreasing in $a$ in the domain $\set{a<b} \subset \Real^2$.
\end{lem}
\begin{proof}
Clearly, it is enough to prove the claim for $[a,b]$ in the interval supporting $f$. 
Note that since $t \mapsto f(-t)$ is also log-concave, it is enough to prove the claim just for the upper limit $b$. 
Translating, we may assume that $a=0$, and so $b > 0$. 
Set $F(t) = \int_0^t f(s) ds$, $F_\infty = \int_0^\infty f(s) ds$, and $I = f \circ F^{-1} : [0,F_\infty] \rightarrow \Real_+$. By definition:
\[
\J(f,[0,b])(v) = \min\brac{\frac{I( v \int_0^b f(s) ds)}{\int_0^b f(s) ds} ,\frac{I((1-v) \int_0^b f(s) ds)}{\int_0^b f(s) ds}} ~,
\]
so it is enough to prove the claim just for the first term inside the minimum above. Assuming that $f$ is smooth on its support (the general case follows by approximation), direct differentiation of this term in $b$ reveals that it is enough to show that:
\[
I'(x) \leq \frac{I(x)}{x} \;\;\; \forall x \in (0,F_\infty) ~.
\]
But note that $I(x) I''(x) = (\log f)''(F^{-1}(x)) \leq 0$, and since $I > 0$ on $(0,F_\infty)$, it follows that $I$ is concave on $[0,F_\infty]$. Concavity directly implies that:
\[
I'(x) \leq \frac{I(x) - I(0)}{x} \leq \frac{I(x)}{x} \;\;\; \forall x \in (0,F_\infty)~,
\]
as required, completing the proof.
\end{proof}

\begin{proof}[Proof of Corollary \ref{cor:main1}]

First, assume that $q<\infty$. We set $m := n+q-1$, recall that $\delta := \rho / m$, and if $\rho \neq 0$ denote:
\[
\beta := \frac{H}{m \sqrt{\abs{\delta}}} ~.
\]

\noindent
\textbf{Cases 1 and 2.} Assume in addition that $\rho > 0$, and observe that:
\[
J_{H,\rho,m}(t) = \brac{\cos(\sqrt{\delta}t) + \beta \sin(\sqrt{\delta}t)}_+^{m} = \brac{\frac{\sin(\alpha + \sqrt{\delta}t)}{\sin(\alpha)}}_+^{m} ~,
\]
where:
\[
\alpha := \cot^{-1}\brac{\beta} \in (0,\pi) ~.
\]
It follows immediately that:
\[
\inf_{H \in \Real , a \in [D-D,D]} \J(J_{H,\rho,n+q-1}(t),[-a,D-a]) = \inf_{\alpha \in (0,\pi) , a \in [D-D,D] } \J(\sin(\alpha + \sqrt{\delta} t)_+^{n+q-1},[-a,D-a]) ~.
\]
Performing the change of variables $\xi = \alpha / \sqrt{\delta} + t$, it follows that:
\[
\inf_{H \in \Real , a \in [D-D,D]} \J(J_{H,\rho,n+q-1}(t),[-a,D-a]) = \inf_{\xi \in [-D,D-D+\pi/\sqrt{\delta}]} \J(\sin(\sqrt{\delta} t)_+^{n+q-1},[\xi,\xi+D]) ~.
\]
Finally, observe that the function $t \mapsto \sin(\sqrt{\delta} t)_+^{n+q-1}$ is log-concave, and hence Lemma \ref{lem:profile-decreasing} implies that the worst case on the right-hand side above is when the model density has maximal support, so that:
\[
\inf_{H \in \Real , a \in [D-D,D]} \J(J_{H,\rho,n+q-1}(t),[-a,D-a]) = \inf_{\xi \in [0,\max(\pi/\sqrt{\delta}-D,0)]} \J(\sin(\sqrt{\delta} t)^{n+q-1},[\xi,\xi+D]) ~,
\]
as asserted in Cases 1 and 2.

\noindent
\textbf{Case 3.} Assume in addition that $\rho=0$ and $D<\infty$. The first assertion then follows by taking the limit as $\rho \rightarrow 0$ in Case 1, but this requires justification. We prefer to deduce the assertion directly. Indeed, note that:
\[
J_{H,0,m}(t) = \brac{1 + \frac{H}{m} t}_+^{m} ~.
\]
Observe that when $H=0$ we obtain the uniform density, and so $\J(J_{0,0,m}(t),[-a,D-a])(v) = \frac{1}{D} = \J(1,[0,D])(v)$ for all $v \in (0,1)$ and $a \in \Real$. When $H \neq 0$, we may translate by setting $s = t + \frac{m}{H}$, obtaining:
\begin{eqnarray*}
\inf_{H \in \Real \setminus \set{0}, a \in [0,D]} \J(J_{H,0,m}(t),[-a,D-a]) & = & \inf_{H \in \Real \setminus \set{0}, a \in [0,D]} \J(s^{m}_+ , \left [\frac{m}{H}-a , \frac{m}{H}+D-a \right]) \\
&=& \inf_{\xi \in \Real} \J(s^{m}_+,[\xi,\xi+D]) ~.
\end{eqnarray*}
Since $s \mapsto s^{m}_+$ is log-concave, Lemma \ref{lem:profile-decreasing} implies that it is enough to test $\xi \geq 0$ above, as asserted in Case 3. 
In fact, an elementary calculation reveals that pointwise:
\[
\lim_{\xi \rightarrow \infty} \J(s^{m},[\xi,\xi+D]) = \J(1,[0,D]) ~,
\]
and so we conclude that:
\begin{eqnarray*}
\inf_{H \in \Real, a \in [0,D]} \J(J_{H,0,n+q-1}(t),[-a,D-a])  &=& 
\min \left \{ \begin{array}{l}  \inf_{\xi \geq 0} \J( s^{n+q-1} , [\xi,\xi+D] ) ~,\\
\phantom{\inf_{\xi \in \Real}} \J(1,[0,D])
\end{array}
\right \}  \\
& = &  \inf_{\xi \geq 0} \J(s^{n+q-1},[\xi,\xi+D]) ~.
\end{eqnarray*}
The second assertion follows by direct calculation. It is clear that the uniform density in the formulation of the lower bound given in Case 3 was only added for completeness of the description of all model densities. 

\noindent
\textbf{Case 4.} Assume in addition that $\rho<0$ and $D<\infty$, and observe that:
\[
J_{H,\rho,m}(t) = \brac{\cosh(\sqrt{-\delta}t) + \beta \sinh(\sqrt{-\delta}t)}_+^{m} =
\begin{cases} \brac{\frac{\sinh(\alpha + \sqrt{-\delta}t)}{\sinh(\alpha)}}_+^{m} & \abs{\beta} > 1 \\
\brac{\frac{\cosh(\alpha + \sqrt{-\delta}t)}{\cosh(\alpha)}}^{m} & \abs{\beta} < 1 \\
\exp(\sqrt{-\delta} m t) & \beta = 1 \\
\exp(-\sqrt{-\delta} m t) & \beta = -1
\end{cases} ~,
\]
where:
\[
\alpha := \begin{cases} \coth^{-1}(\beta) \in \Real \setminus \set{0} & \abs{\beta} > 1 \\ \sinh^{-1}(\beta) \in \Real & \abs{\beta} < 1 \end{cases} ~.
\]
It easily follows that:
\[
\inf_{H \in \Real , a \in [0,D]} \J(J_{H,\rho,n+q-1}(t),[-a,D-a]) = \min \left \{ \begin{array}{l}
\inf_{\xi \in \Real} \J( \sinh(\sqrt{-\delta} t)_+^{n+q-1}, [\xi,\xi+D] ) ~ , \\
\inf_{\xi \in \Real} \J( \exp(\sqrt{-\delta} (n+q-1) t) , [\xi,\xi+D] ) ~, \\
\inf_{\xi \in \Real} \J(\cosh(\sqrt{-\delta} t)^{n+q-1},[\xi,\xi+D] )
\end{array}
\right \} ~.
\]
Observing that the function $t \mapsto \sinh(t)_+$ is log-concave and employing Lemma \ref{lem:profile-decreasing}, it follows that the first infimum in the right-hand-side above need only be taken over $\set{ \xi \geq 0}$. By scale invariance of the exponential function, the second infimum need only be tested at $\xi=0$. The assertion of Case 4 follows.

\medskip

Assume now that $q = \infty$, and recall that:
\[
J_{H,\rho,\infty}(t) = \exp(H t - \frac{\rho}{2} t^2) ~.
\]

\noindent \textbf{Cases 5 and 6.} Assume in addition that $\rho \neq 0$. Performing the change of variables $s = t - H / \rho$, it follows that:
\begin{eqnarray*}
& & \inf_{H \in \Real , a \in [D-D,D]} \J(J_{H,\rho,\infty}(t),[-a,D-a]) \\
&=& \inf_{H \in \Real , a \in [D-D,D]} \J(\exp(-\frac{\rho}{2} s^2),[-a-H/\rho,D-a-H/\rho]) \\
&=& \begin{cases} \inf_{\xi \in \Real} \J(\exp(-\frac{\rho}{2} s^2),[\xi,\xi+D]) & D < \infty \\ \phantom{\inf_{\xi \in \Real}} \J(\exp(-\frac{\rho}{2} s^2),\Real) & D = \infty \end{cases} ~,
\end{eqnarray*}
and the assertions of Cases 5 and 6 follow.

\noindent \textbf{Case 7.} Assume in addition that $\rho = 0$ and $D<\infty$. It follows immediately from the invariance of the exponential function under scaling and of $\J$ under reflection that:
\[
\inf_{H \in \Real , a \in [0,D]} \J(J_{H,0,\infty}(t),[-a,D-a]) = \inf_{H \in \Real} \J(\exp(H t),[0,D]) = \inf_{H \geq 0} \J(\exp(H t),[0,D]) ~,
\]
and the first assertion of Case 7 follows. Note that this also follows by taking the limit as $\rho \rightarrow 0$ in Case 5 (after scaling and translating the density to be 1 at the origin), but this is not so transparent and in any case requires justification.
The second assertion follows by direct calculation.
\end{proof}

\section{Sharpness} \label{sec:sharp}

In this section, we provide a proof of Theorem \ref{thm:main2}.

We would like to show that the bounds provided in Corollary \ref{cor:main1} are pointwise sharp. These bounds are all of the form:
\[
 \I(M,g,\mu)(v) \geq \inf_{\sigma \in \Sigma} \J(f_\sigma,L_\sigma)(v) \;\;\; \forall v \in [0,1] ~.
\]
Fixing $\sigma \in \Sigma$ and $v \in (0,1)$, we will construct a family indexed by $\eps > 0$ of $n$-dimensional manifolds-with-density $(M_\eps,g_\eps,\mu_\eps)$ satisfying the $CDD(\rho,n+q,D)$ condition, and find Borel test sets $A_{\eps} \subset M_\eps$ so that $\mu_\eps(A_{\eps}) = v$ and $\lim_{\eps \rightarrow 0} \mu_\eps^+(A_{\eps}) \leq \J(f_{\sigma},L_{\sigma})(v)$.

\medskip

First, observe that when $\rho \leq 0$ and $D = \infty$, the right-hand side of (\ref{eq:main1}) is 0 by the non-integrability of the stated density, and that this is indeed the best isoperimetric inequality one can hope for under the $CDD(\rho,n+q,D)$ condition. To see this, note that by scaling the metric by a factor of $\lambda^2$, if $(M,g,\mu)$ satisfies the $CDD(\rho,n+q,D)$ condition then $(M,\lambda^2 g,\mu)$ satisfies the $CDD(\rho/\lambda^2,n+q,\lambda D)$ condition. Consequently, when $\rho \leq 0$ and $D=\infty$, if $(M,g,\mu)$ satisfies the $CDD(\rho,n+q,\infty)$ condition then so does $(M,\lambda^2 g ,\mu)$ when $\lambda \geq 1$. However, it follows from the definition of boundary measure that $\I(M,\lambda^2 g,\mu) = \frac{1}{\lambda} \I(M,g,\mu)$, and so tending $\lambda \rightarrow \infty$, we see that $0$ is indeed the best possible lower bound on the isoperimetric profile in this case.

Next, we treat the easy case of $q=\infty$. The well known sharpness of Case 6 follows immediately by considering the space $(\Real^n,|\cdot|,\gamma^\rho_n)$ and taking the test set $A$ to be a half-plane (there is no need to use an approximating sequence here). Let us therefore concentrate on Case 5, as Case 7 follows similarly.

\begin{proof}[Proof of Sharpness of Case 5]
Let $\xi_0 \in \Real$ and $v \in (0,1)$.
We consider Euclidean space $(\Real^n,|\cdot|)$, and given $\eps > 0$, set $\Omega_\eps = [\xi_0,\xi_0 + \sqrt{D^2 - (n-1) \eps^2}] \times [0,\eps]^{n-1}$, having diameter $D$. Note that $\Omega_\eps$ does not have a smooth boundary, but this can be fixed by taking an additional approximation by convex smooth domains with the same bound on their diameter (see also Section \ref{sec:gen} for more on approximation). Define $\mu_\eps$ to be the probability measure on $\Omega_\eps$ having density proportional to $\exp(-\frac{\rho}{2}|x|^2)$, and note that $(\Real^n,|\cdot|,\mu_\eps)$ satisfies the $CDD(\rho,\infty,D)$ condition. Now let $A_{\eps}^{-} = \set{x_1 \leq a^{-}}$ and $A_{\eps}^+ = \set{x_1 \geq a^+}$ so that $\mu_\eps(A_{\eps}^{\pm}) = v$, and set $A_{\eps}$ to be $A_{\eps}^{-}$ or $A_{\eps}^{+}$ according to whichever half-plane has smaller $\mu_\eps$-boundary measure. Clearly the product structure ensures that:
\[
(\mu_\eps)^+(A_{\eps}) = \J(\exp(-\frac{\rho}{2} t^2),[\xi_0,\xi_0+\sqrt{D^2 - (n-1) \eps^2}])(v)
\]
(or only approximately when using the approximation by smooth domains). Taking the limit as $\eps$ goes to $0$, it follows that the lower bound of Case 5 cannot be pointwise improved.
\end{proof}

When $\rho = 0$, the case $q < \infty$ follows along the same lines:

\begin{proof}[Proof of Sharpness of Case 3]
Let $\xi_0 \geq 0$ and $v \in (0,1)$.
Consider again Euclidean space $(\Real^n,|\cdot|)$, and given $\eps>0$, set $\Omega_\eps$ to be the truncated cone:
\[
\set{(x_1,y) \in \Real \times \Real^{n-1} ; x_1 \in [\xi_0,\xi_0+D_\eps], |y| \leq \eps x_1} ~,
\]
where $D_\eps$ is chosen so that the $\Omega_\eps$ has diameter $D$ (obviously $\lim_{\eps \rightarrow 0} D_\eps = D$). As before, this truncated cone is not smooth at either of its sides, but may be approximated by convex smooth domains with the same bound on their diameter
(see also Section \ref{sec:gen} for more on approximation). Denote by $\mu_\eps$ the probability measure on $\Omega_\eps$ with density proportional to $(x_1)^q$, and note that $(\Real^n,|\cdot|,\mu_\eps)$ satisfies the $CDD(0,n+q,D)$ condition. Now let $A_{\eps}$ be a half-plane of the form $\set{x_1 \leq a}$ in case $v \leq 1/2$ and $\set{x_1 \geq a}$ otherwise, where $a$ is chosen so that $\mu(A_{\eps}) = v$. It follows that:
\[
 \mu^+(A_{\eps}) = \J(t^{n+q-1},[\xi_0,\xi_0+D_\eps])(v)
\]
(or only approximately when using the approximation by smooth domains). Taking the limit as $\eps$ goes to $0$, it follows that the lower bound of Case 3 cannot be pointwise improved, as asserted.
\end{proof}

The other cases when $q < \infty$ pose a bigger challenge. We will simultaneously handle Cases 1 and 4, Case 2 follows from Case 1 by approximation.

\begin{proof}[Proof of Sharpness when $q < \infty$]
Assume that $q > 0$, the case that $q=0$ follows either by approximation or by an argument which is actually simpler than the one described below (at least when $n \geq 3$). Let $\rho, H \in \Real$, $D \in (0,\infty]$, $a \in [D-D,D]$, and set $J = J_{H,\rho,n+q-1}$. Given $v \in (0,1)$, we would like to construct an oriented manifold $M$ endowed with smooth complete Riemannian metrics $\set{g_\eps}$ and probability measures $\set{\mu_\eps}$, so that each $(M,g_\eps,\mu_\eps)$ satisfies the $CDD(\rho,n+q,D)$ condition (whenever $\eps > 0$ is small enough) and so that $\lim_{\eps \rightarrow 0} \I(M,g_\eps,\mu_\eps)(v) \leq \J(J,[-a,b])(v)$ with $a+b = D$. Since $\J(J,[-a,b])(v)$ is continuous in $a,b \in \Real$ and does not change when $a$ and $b$ vary outside the support of $J$, we may assume that $J(-a),J(b) > 0$ and that $a+b < D$.

We construct the $n$-dimensional manifold $M := T_\infty \times S^{n-1}$ with $T_\infty := [e^1,e^2] \supset T := [-a,b]$, as described next. Given $\eps > 0$, we equip $M$ with the metric $g_\eps$ given by:
\[
g_\eps := dt^2 + \varrho_\eps(t)^2 g_{S^{n-1}} ~,~ \varrho_\eps(t) := \eps f_\eps(t) ~.
\]
Here $g_{S^{n-1}}$ denotes the standard metric on $S^{n-1}$, and $f_\eps: T_\infty \rightarrow \Real_+$ is a smooth (uniformly bounded in $\eps$) function to be determined later.
The probability measure $\mu_\eps$ will be supported on the set $T_\eps \times S^{n-1}$, where $T_\eps := [-a-\omega_1(\eps),b+\omega_2(\eps)]$, and $\omega_1(\eps),\omega_2(\eps) \geq 0$ are small constants tending to $0$ as $\eps \rightarrow 0$ to be determined later on.
Since $a + b < D$, when $\eps > 0$ is small enough, the latter's diameter will clearly be at most $D$.
We specify $\mu_\eps$ by setting:
\[
\mu_\eps := \Psi_\eps vol_{g_\eps}|_{T_\eps \times S^{n-1}} ~,~ \Psi_\eps^{1/q}(t,\theta) = c_\eps \g_\eps(t) ~,~ \text{for} \;\; (t,\theta) \in T_\eps \times S^{n-1} ~,
\]
with $c_\eps > 0$ a normalization constant, and $\g_\eps: T_\eps \rightarrow (0,\infty)$ a smooth function to be determined.

\medskip
\noindent
\textbf{Intuition.} Set:
\[
f_\eps(t) = c_J(\eps) J_0(t)  ~,~  \g_\eps(t) = J_0(t) ~,~  J_0(t) := J(t)^{\frac{1}{n+q-1}} ~,~ \text{for}~ t \in T ~,
\]
where $c_J(\eps)$ is some parameter we need for technical reasons, depending on $\eps$ and satisfying $0 < c_J^1 \leq c_J(\eps) \leq c_J^2 < \infty$.
The intuition behind this construction is that when $\eps > 0$ is small enough, the geometry of $(M,g_\eps)$ will contribute (at least) $(n-1)\delta$ to the generalized Ricci curvature and a factor of $J_0^{n-1}$ to the density $d\mu_{\eps} \brac{(-\infty,t] \times S^{n-1}}/dt$, whereas the measure $\mu_\eps$ will contribute $q \delta$ to the former and a factor of $\Psi = J_0^q$ to the latter, totalling $(n-1+q)\delta = \rho$ and $J_0^{n-1+q} = J$, respectively. We will indeed verify below that $Ric_{g_\eps,\Psi_\eps,q} \geq \rho g_\eps$ on the set $\Omega := T \times S^{n-1}$, if $n \geq 3$ and $\eps>0$ is small enough. What prevents us from setting $\omega_1(\eps) = \omega_2(\eps) = 0$ and concluding that the $CDD(\rho,n+q,D)$ condition is satisfied on $\Omega$, is that the latter will not be geodesically convex in general.

\medskip
\noindent
\textbf{Geodesic Convexity.}
Indeed, let us first check the second fundamental form of $\partial \Omega \subset (M,g_\eps)$ assuming $\omega_1(\eps) = \omega_2(\eps) = 0$.
Given $x = (t,\theta) \in T_\infty \times S^{n-1}$, let $\partial_t,\partial_{\theta_1},\ldots,\partial_{\theta_{n-1}}$ denote an orthonormal basis in $T_x M$. An elementary computation verifies that the second fundamental form of the submanifold $\set{t} \times S^{n-1}$ with respect to the normal $\partial_t$ (and our convention for specifying its sign from Section \ref{sec:pre}) is given by $(\log \varrho_\eps)'(t)$ times the identity on the submanifold's tangent space.
Consequently, if $J_0'(-a) \leq 0$ ($J_0'(b) \geq 0$), then the left (right) boundary of $T \times S^{n-1}$ has non-negative second fundamental form with respect to the outer normal, which is known \cite{BishopInBayleRosales} to be equivalent to local geodesic convexity near that boundary. Otherwise, if $J_0'(-a) > 0$ ($J_0'(b) < 0$), we will choose a metric $g_\eps$ which closes up our manifold near $-a$ ($b$) into a small smooth cap, so that $T_\eps \times S^{n-1}$ does not have a boundary there; we will refer to these terminal points as \emph{vertices.}

We therefore set:
\begin{equation} \label{eq:vertex}
e^1 := \begin{cases} -\infty & J'(-a) \leq 0 \\ -a-\omega_1(\eps) & J'(-a) > 0 \end{cases} ~ , ~
e^2 := \begin{cases} \infty & J'(b) \geq 0 \\ b+\omega_2(\eps) & J'(b) < 0 \end{cases} ~.
\end{equation}
In order to obtain a smooth manifold at the vertex $e^1$ ($e^2$) in case the bottom possibility above occurs, it is well known (e.g. \cite[p. 13]{PetersenBook2ndEd}) that we need to require that $\varrho_\eps^{(2k)}(e^i) = 0$, for all non-negative integers $k$, and that $\varrho'_\eps(e^1) = 1$ ($\varrho'_\eps(e^2) = -1$). Consequently, we will make sure that:
\begin{equation} \label{eq:cond1}
f_\eps^{(2k)}(e^i) = 0 ~,~ f'_\eps(e^1) = 1/\eps ~ ~ (f'_\eps(e^2) = -1/\eps) ~.
\end{equation}
Furthermore, to ensure that we obtain a smooth density at the vertex $e^i$, we will force $\g_\eps$ to be constant near the vertex.

When $\set{-a} \times S^{n-1}$ ($\set{b} \times S^{n-1}$) has non-negative second fundamental form (given by the top possibility in (\ref{eq:vertex})), we will need to extend the local geodesic convexity near this submanifold to a global one. To this end, we simply make sure to extend $f_\eps$ smoothly and \emph{monotonically} on $(-\infty,-a]$ ($[b,\infty)$). Indeed, any continuous path exiting $T \times S^{n-1}$ at e.g. $(b,\theta_1)$, will have to return to this set at $(b,\theta_2)$; however, since by construction $f_\eps(t) \geq f_\eps(b)$ if $t \geq b$, it follows by projecting onto $\set{b} \times S^{n-1}$ that the path cannot be shorter than the path $s \mapsto (b,\gamma(s))$, where $\gamma$ is a geodesic on $S^{n-1}$ connecting $\theta_1$ and $\theta_2$, and so geodesic convexity is established.

\medskip
\noindent
\textbf{Curvature Calculation.}
Clearly, our orthonormal basis $\partial_t,\partial_{\theta_1},\ldots,\partial_{\theta_{n-1}}$ diagonalizes both $Ric_{g_\eps}$ and $\nabla^2_{g_\eps} \Psi^{1/q}$. Moreover, since $\Psi_\eps^{1/q}(t,\theta) = c_\eps \g_\eps(t)$ depends on $t$ only, we easily verify that at $x = (t,\theta)$:
\[
\nabla^2_{g_\eps} \Psi_\eps^{1/q}(\partial_t,\partial_t) = c_\eps \g_\eps''(t) ~,~
\nabla^2_{g_\eps} \Psi_\eps^{1/q}(\partial_{\theta_i},\partial_{\theta_i}) = c_\eps \frac{\varrho'_\eps(t)}{\varrho_\eps(t)} \g_\eps'(t) ~.
\]
It is known (e.g. \cite[p. 68]{PetersenBook2ndEd}) that for rotationally invariant metrics such as $g_\eps$, the sectional curvature in $2$-planes containing $\partial_t$ is given by $-\varrho_\eps''(t)/\varrho_\eps(t)$, and in $2$-planes orthogonal to $\partial_t$ by $(1-\varrho_\eps'(t)^2) / \varrho_\eps(t)^2$. Recalling that:
\[
Ric_{g_\eps,\Psi_\eps,q} := Ric_{g_\eps} - q \frac{\nabla^2_{g_\eps} \Psi_\eps^{1/q}}{\Psi_\eps^{1/q}}
\]
and putting everything together, we obtain that:
\[
Ric_{g_\eps,\Psi_\eps,q}(\partial_t,\partial_t) =
-(n-1) \frac{f''_\eps(t)}{f_\eps(t)} - q \frac{\g_\eps''(t)}{\g_\eps(t)} ~;
\]
\begin{eqnarray*}
Ric_{g_\eps,\Psi_\eps,q}(\partial_{\theta_i},\partial_{\theta_i}) & = & - \frac{\varrho''_\eps(t)}{\varrho_\eps(t)} + (n-2) \frac{1 - \varrho_\eps'(t)^2}{\varrho_\eps(t)^2}  - q \frac{\varrho'_\eps(t)}{\varrho_\eps(t)} \frac{\g_\eps'(t)}{\g_\eps(t)} \\
&=&
- \frac{f''_\eps(t)}{f_\eps(t)} + (n-2) \frac{1 - \eps^2 f_\eps'(t)^2}{\eps^2 f_\eps(t)^2}  - q \frac{f'_\eps(t)}{f_\eps(t)} \frac{\g_\eps'(t)}{\g_\eps(t)} ~.
\end{eqnarray*}
Recall by Remark \ref{rem:J-char} that on $[-a,b]$, $J_0$ satisfies:
\[
J_0'' + \delta J_0 = 0 ~ , ~ \delta := \frac{\rho}{n+q-1} ~,
\]
and so on $T \times S^{n-1}$, it easily follows (see the subsequent calculation) that when $n \geq 3$ and for $\eps>0$ small enough, $Ric_{g_\eps,\Psi_\eps,q} \geq ((n-1) \delta + q \delta) g_\eps = \rho g_\eps$.

\medskip
\noindent
\textbf{Gluing Caps.}
It remains to properly handle the end points $-a$ and $b$. If $e^i$ is not a vertex point we simply set $\omega_i(\eps) = 0$. If on both sides we have no vertices then this concludes the construction (without taking any limit in $\eps$) - this may happen when $\rho \leq 0$ and for certain values of $H$ and $a$, as apparent in Case 3 and some of the subcases of Case 4.
However, in the presence of a vertex at $e^i$, setting $\omega_i(\eps) = 0$ is forbidden since this would be in violation of (\ref{eq:cond1}), rendering the manifold non-smooth at the vertex; and even if it were smooth, the density $\Psi_\eps$ would fail to be smooth there. To work around this problem, we ``glue" arbitrarily small smooth caps to $\Omega_\eps$ and endow them with an appropriate density, by appropriately defining $f_\eps$ and $\g_\eps$ on $T \setminus T_\eps$, in a manner ensuring that the curvature condition remains valid there, as described next.

Let us assume for simplicity that we only have one vertex at $e^1$ and describe the construction on $[-a-\omega_1(\eps),b]$; the required modifications on $[b,b+\omega_2(\eps)]$ in the case that $e^2$ is also a vertex are completely analogous. For some constants $0<\alpha(\eps)<\beta(\eps)<\omega(\eps) = \omega_1(\eps)$ to be determined and tending to $0$ as $\eps \rightarrow 0$, we set:
\[
f_\eps(t+a) := \begin{cases}
\sin(\frac{1}{\eps} (t + \omega(\eps))) & t \in [-\omega(\eps),-\beta(\eps)] \\
\sin(\frac{1}{\eps} (t + \omega(\eps)))  & t \in [-\beta(\eps),-\alpha(\eps)] \\
\Phi_\eps(t) & t \in [-\alpha(\eps) , 0] \\
c_J(\eps) J_0(t) & t \in [0,b+a]
\end{cases} ~,~
\g_\eps(t+a) := \begin{cases}
c_\eps & t \in [-\omega(\eps),-\beta(\eps)]  \\
\Gamma_\eps(t) & t \in [-\beta(\eps),-\alpha(\eps)]  \\
J_0(t) & t \in [-\alpha(\eps) , 0] \\
J_0(t) & t \in [0,b+a]
\end{cases} ~;
\]
the functions $\Phi_\eps$ and $\Gamma_\eps$ smoothly interpolate between the corresponding functions above in a manner described next.
Since we assume that $J_0'(-a), J_0(a) > 0$, we can make sure that $J_0'(-a) / 2 \leq J_0'(-a - \alpha(\eps)) \leq 2 J_0'(-a)$ and $J_0(-a) / 2 \leq J_0(-a - \alpha(\eps)) \leq 2 J_0(-a)$ if $\eps >0$ is small enough. For $\eps > 0$ small enough, set:
\[
c_\eps := J_0(-a-\alpha(\eps)) - \frac{1}{3} J_0'(-a-\alpha(\eps)) (\beta(\eps)-\alpha(\eps)) \geq J_0(-a) / 4 =: c_\Gamma^1 > 0 ~;
\]
it is easy to see that we may then choose $\Gamma_\eps$ to smoothly interpolate between $c_\eps$ and $J_0(t)$ so that it satisfies $0 < \Gamma'_\eps \leq 10 J_0'(-a)$ and $|\Gamma''_\eps| \leq 100 J_0'(-a) / (\beta(\eps) - \alpha(\eps))$.
Lemma \ref{lem:Phi} below ensures that setting $\alpha(\eps) = C \eps$, for some small enough $C>0$, $\Phi_\eps$ may be chosen to smoothly interpolate between $\sin(\frac{1}{\eps} (t + \omega(\eps)))$ and $c_J(\eps) J_0(t)$ for an appropriate $0 < c_J^1 \leq c_J(\eps) \leq c_J^2 < \infty$, so that $0 < \Phi'_\eps \leq 1/(2\eps)$ and $\Phi''_\eps / \Phi_\eps \leq -\delta$, and so that $(\pi/4) \eps \leq \omega(\eps)-\alpha(\eps) \leq (\pi/2) \eps$.
It follows that $c^1_\Phi := \sin(\pi/4) \leq \Phi_\eps \leq c^2_\Phi := c_J^2 J_0(-a)$. Setting $\beta(\eps) = (\omega(\eps) + \alpha(\eps))/2$, so that $\omega(\eps) - \beta(\eps) \geq (\pi/8) \eps$, our construction is complete.
Putting it all together, we obtain for $\eps > 0$ small enough:
\begin{equation} \label{eq:conc1}
Ric_{g_\eps,\Psi_\eps,q}(\partial_t,\partial_t) \geq \begin{cases} \frac{n-1}{\eps^2} - \frac{q 100 J_0'(-a)}{\eps \pi/8 } & t \in [-a-\omega(\eps), -a - \alpha(\eps)] \\
(n-1) \delta + q \delta = \rho & t \in [-a - \alpha(\eps) , b] \end{cases} ~;
\end{equation}
\begin{equation} \label{eq:conc2}
Ric_{g_\eps,\Psi_\eps,q}(\partial_{\theta_i},\partial_{\theta_i}) \geq \begin{cases}
\frac{1}{\eps^2} + \frac{n-2}{\eps^2} & t \in [-a - \omega(\eps), -a - \beta(\eps)] \\
\frac{1}{\eps^2} + \frac{n-2}{\eps^2} - q \frac{\cot(\pi/8)}{\eps} \frac{10 J'_0(-a)}{c_\Gamma^1}  & t \in [-a -\beta(\eps) , -a -\alpha(\eps)] \\
\delta + (n-2) \frac{1-1/4}{\eps^2 c_\Phi^2} - q \frac{1/(2\eps)}{c_\Phi^1} M_2 & t \in [-a - \alpha(\eps) , -a] \\
\delta + (n-2) \frac{1 - \eps^2 (c_J^2)^2 M_1^2 M_2^2}{\eps^2 (c_J^2)^2 M_1^2 } - q M_2^2 & t \in [-a,b]
\end{cases} ~,
\end{equation}
where:
\[
M_1 = \max_{t \in [-a,b]} J_0(t) ~,~ M_2 = \max_{t \in [-a,b]} | (\log J)'(t) | ~.
\]
Consequently, when $\eps$ tends to $0$, the quadratic terms in $1/\eps$ appearing in (\ref{eq:conc1}) and (\ref{eq:conc2}) dominate over the linear ones, and we readily verify that the $CDD(\rho,n+q,D)$ condition holds for small enough $\eps > 0$ on the smooth manifold-with-density $(M,g_\eps,\mu_\eps)$, when $n \geq 3$.

\medskip
\noindent
\textbf{Two-Dimensional Case.} When $n=2$, we slightly modify our construction as follows. Identifying $S^1$ with $[-\pi,\pi]$, we first restrict to the set $\Omega'_\eps := T_\eps \times [-\pi/2,\pi/2]$. By the symmetry, it is clear since $\Omega_\eps$ is geodesically convex that $\Omega'_\eps$ is too. We now modify the probability measure $\mu_\eps = \Psi_\eps vol_{g_\eps}|_{\Omega'_\eps}$, as follows:
\[
\Psi_\eps^{1/q}(t,\theta) = c_\eps \g_\eps(t) \nu_\eps(t,\theta)  ~,~ \nu_\eps(t,\theta) = \cos(h_\eps(t)\theta) ~,
\]
for $x = (t,\theta) \in T_\eps \times [-\pi/2,\pi/2]$. First, we set:
\[
h_\eps(t) = C_h  \text{ for } t \in [-a,b] ~,
\]
where $C_h>0$ is some small enough constant. To describe $h_\eps$ on $T_\eps \setminus T$, let us as before assume for simplicity that we only have a single vertex on the left, and set:
\[
h_\eps(t) := \begin{cases} 0 & t \in [-a -\omega(\eps),-a-\beta(\eps)] \\
\chi_\eps(t) & t \in [-a-\beta(\eps),-a-\alpha(\eps)] \\
C_h & t \in [-a - \alpha(\eps) , b]
\end{cases} ~,
\]
where $\chi_\eps$ increases smoothly from $0$ to $C_h$ in a manner so that $\chi'_\eps(t) \leq 10 C_h / (\beta(\eps) - \alpha(\eps))$ and $|\chi''(t)| \leq 100 C_h / (\beta(\eps) - \alpha(\eps))^2$. Note that the resulting $\Psi_\eps^{1/q}$ is smooth thanks to the restriction to $\Omega'_\eps$ and the fact that $h_\eps$ is $0$ in a neighborhood of the vertex.

Unfortunately, our orthonormal basis $\partial_t,\partial_{\theta_1}$ no longer diagonalizes $\nabla^2_{g_\eps} \Psi_\eps^{1/q}$, but it is still possible to verify that $Ric_{g_\eps,\Psi_\eps,q} \geq \rho g_\eps$ when $\eps>0$ is small enough, since this amounts to checking that a 2 by 2 matrix is positive-definite. We omit the extremely tedious computation, but only remark that the role of $\nu_\eps$ is to add ``more generalized curvature" in the $\partial_{\theta_1}$ direction, a point which we could avoid when $n \geq 3$. To summarize, the $CDD(\rho,n+q,D)$ condition holds for small enough $\eps > 0$ on $(M,g_\eps,\mu_\eps)$ as well.

\medskip
\noindent \textbf{Verifying Sharpness.} Now let $A_{\eps}^{-} = \set{(t,\theta) \in M ; t \leq t_0 }$ and $A_{\eps}^+ = \set{(t,\theta) \in M ; t \geq t_0 }$ so that $\mu_\eps(A_{\eps}^{\pm}) = v$, and set $A_{\eps}$ to be $A_{\eps}^{-}$ or $A_{\eps}^{+}$ according to whichever set has smaller $\mu_\eps$-boundary measure. When $n \geq 3$, our construction ensures that:
\[
(\mu_\eps)^+(A_{\eps}) = \J(f_\eps(t)^{n-1} \g_\eps(t)^{q},[-a-\omega_1(\eps),b+\omega_2(\eps)])(v) ~.
\]
Taking the limit as $\eps \rightarrow 0$, since $c_J(\eps) \geq c_J^1 > 0$, the right-hand side above tends to the desired $\J(J_0^{n+q-1},[-a,b])(v)$, and hence it follows that the lower bound given by Theorem \ref{thm:main1} when $q < \infty$ cannot be pointwise improved. When $n=2$, we obtain:
\[
(\mu_\eps)^+(A_{\eps}) = \J(f_\eps(t) \g_\eps(t)^{q} z_\eps(t) ,[-a-\omega_1(\eps),b+\omega_2(\eps)])(v) ~,
\]
with $z_\eps(t) := \int_{-\pi/2}^{\pi/2} \cos(h_\eps(t) \theta)^q d\theta$, and since the latter is a constant function on $[-a,b]$, the desired sharpness follows similarly by taking the limit as $\eps \rightarrow 0$.
\end{proof}

It remains to establish:

\begin{lem} \label{lem:Phi}
Let $\delta \in \Real$, and let $J_\delta$ denote a smooth function on $\Real$ satisfying $J_\delta''(t) + \delta J_\delta(t) = 0$ with $J_\delta(0),J_\delta'(0) > 0$.
Then for any $\eps > 0$ small enough, there exist $\omega(\eps) > \alpha(\eps) > 0$ tending to $0$ as $\eps \rightarrow 0$ and a smooth function $\Phi_\eps$ on $\Real$ so that:
\begin{itemize}
\item
On $(-\infty,-\alpha(\eps)]$, $\Phi_\eps(t) = \sin(\frac{1}{\eps}(t+\omega(\eps)))$.
\item
On $[-\alpha(\eps),0]$:
\begin{itemize}
\item $\Phi''_\eps(t) / \Phi_\eps(t) \leq -\delta$.
\item $0 < \Phi'_\eps(t) \leq \frac{1}{2\eps}$.
\end{itemize}
\item
On $[0,\infty)$, $\Phi_\eps(t) = c_J(\eps) J_\delta(t)$, with $\frac{1}{\sqrt{2}} \leq c_J(\eps) J_\delta(0) \leq \sqrt{2}$.
\item
$\alpha(\eps) = C \eps$ for some constant $C > 0$ and $(\pi/4) \eps \leq \omega(\eps) - \alpha(\eps) \leq (\pi/2) \eps$.
\end{itemize}
\end{lem}

\begin{proof}
Let $\Phi_0 = \Phi_{0,\eps}$ denote a smooth solution on $\Real$ to the following Sturm-Liouville equation:
\[
\Phi_0''(t) + \lambda_\eps(t) \Phi_0(t) = 0 ~,~ \Phi_0(0) = J_\delta(0) ~,~ \Phi_0'(0) = J_\delta'(0) ~,
\]
where $\lambda_\eps$ is a smooth non-increasing function interpolating between the values of $1/\eps^2$ on $(-\infty,-\alpha(\eps)]$ and $\delta$ on $[0,\infty)$ (assuming that $\eps>0$ is small enough), with $\alpha(\eps) = C \eps$ for some constant $C > 0$ to be determined later. This implies that $\Phi_0(t) = c^1_\eps \sin(\frac{1}{\eps} t + c^2_\eps)$ on $(-\infty,-\alpha(\eps)]$ for some constants $c^1_\eps \in \Real$ and $c^2_\eps \in [0,\pi]$.

Similarly, let $J_{1/\eps^2}$ denote a smooth solution on $\Real$ to:
\[
J_{1/\eps^2}''(t) + \frac{1}{\eps^2} J_{1/\eps^2}(t) = 0 ~,~ J_{1/\eps^2}(0) = J_\delta(0) ~,~ J_{1/\eps^2}'(0) = J_\delta'(0) ~.
\]
By the maximum principle, it follows that:
\begin{equation} \label{eq:PhiA}
(\log J_{1/\eps^2})'(t) \geq (\log \Phi_0)'(t) \geq (\log J_\delta)'(t) \;\;\; \forall t \in I_\eps := (-a_\eps,0] ~,
\end{equation}
where $a_\eps > 0$ is defined so that both $J_{1/\eps^2}$ and $J_\delta$ are positive on $I_\eps$; this may be easily verified by checking e.g. that $J_{1/\eps^2}' \Phi_0 - J_{1/\eps^2} \Phi_0'$ is non-increasing on $I_\eps$ and vanishes at the origin. In particular, since all three functions above coincide at the origin, it follows that:
\begin{equation} \label{eq:PhiB}
J_{1/\eps^2}(t) \leq \Phi_0(t) \leq J_\delta(t)  \;\;\; \forall t \in [-\alpha(\eps),0] \cap I_\eps ~.
\end{equation}

Now $J_{1/\eps^2} = d^1_\eps \sin(\frac{t}{\eps} + d^2_\eps)$, and since $d^1_\eps \sin(d^2_\eps) = J_\delta(0) > 0$ and $\frac{1}{\eps} \cot(d^2_\eps) = (\log J_\delta)'(0) > 0$, it immediately follows that for $\eps > 0$ small enough:
\[
d^1_\eps \geq J_\delta(0) > 0 ~,~ \pi/2 - C \leq d^2_\eps < \pi/2 ~.
\]
In particular, we verify that both $J_{1/\eps^2}$ and $J_\delta$ are positive on $[-\alpha(\eps) , 0]$ for $\eps > 0$ small enough if $C = \alpha(\eps)/\eps \leq d^2_\eps$, which is satisfied if we require that $C \leq \pi/4$.

Using (\ref{eq:PhiA}), we deduce if $\eps > 0$ is small enough that:
\begin{equation} \label{eq:PhiC}
0 < (\log J_\delta)'(0) / 2 \leq (\log \Phi_0)'(t) \leq \frac{1}{\eps} \cot(\frac{t}{\eps} + d_\eps^2) \;\;\; \forall t \in [-\alpha(\eps),0] ~.
\end{equation}
Evaluating this at $t = -\alpha(\eps)$, since $(\log \Phi_0)'(-\alpha(\eps)) = \frac{1}{\eps} \cot(c_\eps^2-C)$, we deduce that:
\begin{equation} \label{eq:PhiD}
\pi/2 > c_\eps^2 - C \geq d_\eps^2 - C \geq \pi/2 - 2C \geq \pi/4 ~,
\end{equation}
if we require that $C \leq \pi/8$. Plugging this back into (\ref{eq:PhiC}) implies that:
\[
0 < \max_{t \in [-\alpha(\eps),0]} (\log \Phi_0)'(t) \leq \frac{1}{\eps} \cot(d_\eps^2-C) \leq \frac{1}{\eps} \cot(\pi/2-2C) ~,
\]
and so choosing $C>0$ so that in addition $\cot(\pi/2 - 2C) \leq \frac{1}{2 \sqrt{2}}$, since $\Phi_0$ is increasing on $[-\alpha(\eps),0]$ by (\ref{eq:PhiC}), it follows that:
\[
\Phi_0'(t) \leq \frac{1}{2 \sqrt{2} \eps} \Phi_0(t) \leq \frac{1}{2 \sqrt{2} \eps} J_\delta(0) \;\;\; \forall t \in [-\alpha(\eps),0] ~.
\]

Finally, using (\ref{eq:PhiB}) at $t = -\alpha(\eps)$, we obtain for $\eps > 0$ small enough (since $J_\delta'(0) > 0$):
\[
d^1_\eps \sin(d^2_\eps - C) \leq c^1_\eps \sin(c^2_\eps - C) \leq J_0(0) ~.
\]
By (\ref{eq:PhiD}) and the fact that $d^1_\eps \geq J_0(0)$, the inequalities above easily imply:
\[
\frac{1}{\sqrt{2}} J_\delta(0) \leq c^1_\eps \leq \sqrt{2} J_\delta(0) ~.
\]

Now setting $\Phi_\eps = \Phi_{0,\eps} c_J(\eps)$ with $c_J(\eps) = 1 / c^1_\eps$ and $\omega(\eps) = c^2_\eps \eps$, one verifies that for $C > 0$ small enough, all of the required assertions are satisfied.

\end{proof}

\section{Generalizing the Curvature-Dimension-Diameter Condition} \label{sec:gen}

Before concluding, we slightly generalize the Curvature-Dimension-Diameter condition to dispose of some technical assumptions in the original definition.

Our main motivation for trying to extend the $CDD(\rho,n+q,D)$ condition is the technical requirement that $\Psi > 0$ on the entire $\overline{\Omega}$.
Allowing for $\Psi$ to vanish on $\partial \Omega$ is actually not unnatural, as witnessed by some of our one-dimensional model densities $J_{H,\rho,m}$, which may vanish outside some interval. As already observed in Remark \ref{rem:1D}, we did not require in the one-dimensional case that $\Psi > 0$ on $\partial \Omega$, since this was not needed for the proof of Corollary \ref{cor:main1-1D}. However, our proof for manifolds of arbitrary dimension crucially relied on the known regularity theory for isoperimetric minimizers in the interior of $\Omega$ as well as on its boundary, which as pointed out to us by Frank Morgan, may very well fail in the presence of a density vanishing even at a single point, so we cannot treat this case directly. Instead, we briefly describe an approximation procedure for handling this case, which is also useful for removing some of the other technical assumptions in our original definition of the $CDD$ condition. 
\begin{dfn*}[Generalized Curvature-Dimension-Diameter Condition]
Let $\rho \in \Real$, $q \in [0,\infty]$, $D \in (0,\infty]$.
Assume that $\mu$ may be approximated in total-variation by measures $\set{\mu_m}$ with densities uniformly bounded from above, so that $(M^n,g,\mu_m)$ satisfies the $CDD(\rho_m,n+q_m,D_m)$ condition. Assume that $\rho_m \rightarrow \rho$, $q_m \rightarrow q$ and $D_m \rightarrow D$, as $m \rightarrow \infty$.
We will then say that $(M^n,g,\mu)$ satisfies the generalized $CDD(\rho,n+q,D)$ condition.
\end{dfn*}

\noindent
Recall that $\set{\mu_m}$ is said to converge to $\mu$ in total-variation if:
\[
d_{TV}(\mu_m,\mu) := \sup_{A \subset \Omega} \abs{\mu_m(A) - \mu(A)} \rightarrow_{m \rightarrow \infty} 0 ~.
\]

\begin{prop}
Theorem \ref{thm:main1}, Corollary \ref{cor:main1} and Corollary \ref{cor:main1-1D} continue to hold for any $(M,g,\mu)$ satisfying the generalized $CDD(\rho,n+q,D)$ condition.
\end{prop}
\begin{proof}
According to the proof of \cite[Lemma 6.6]{EMilman-RoleOfConvexity}, for any sequence $\set{\mu_m}$ of Borel probability measures on a common separable metric space $(\Omega,d)$, which tends to $\mu$ in total-variation, we have:
\begin{equation} \label{eq:I-one-sided}
\liminf_{u \rightarrow v} \I_{(\Omega,d,\mu)}(u) \geq \lim_{\eps \rightarrow 0} \limsup_{m \rightarrow \infty} \inf_{|u-v| < \eps} \I_{(\Omega,d,\mu_m)}(u) \;\;\; \forall v \in (0,1) ~.
\end{equation}
Furthermore, it follows from the proof of \cite[Proposition 6.8]{EMilman-RoleOfConvexity}, that if $(M^n,g)$ is a complete smooth oriented connected Riemannian manifold, equipped with a Borel probability measure $\nu$ whose density (with respect to $vol_M$) is bounded above by $C$, then $\I(M,g,\nu): [0,1] \rightarrow \Real_+$ is locally H\"{o}lder continuous, with the modulus of continuity at $v \in (0,1)$ depending solely on $v$, $(M^n,g)$, $n$, $C$, $\delta \in (0,1)$ and an upper bound on the quantity:
\[
R_{v,\delta}(\nu) := \inf \set{R > 0; \nu(B(x_0, R)) \geq 1 - \delta \min(v,1-v)} ~,
\]
where $B(x_0,R)$ denotes the geodesic ball of radius $R$ about a fixed point $x_0 \in M$. Since $\set{\mu_m}$ converge to $\mu$ in total-variation, then $R_{v,1/2}(\mu_m) \leq R_{v,1/4}(\mu)$ for a given $v \in (0,1)$ and $m$ large enough. Together with the fact that the densities of $\set{\mu_m}$ are uniformly bounded above, it follows that the isoperimetric profiles $\I(M,g,\mu_m)$ are locally (in $(0,1)$) uniformly (in $m$) H\"{o}lder continuous. Consequently, (\ref{eq:I-one-sided}) translates in our case to:
\[
\I_{(\Omega,d,\mu)}(v) \geq \limsup_{m \rightarrow \infty} \I_{(\Omega,d,\mu_m)}(v) \;\;\; \forall v \in (0,1) ~.
\]
Applying the lower bound given by Theorem \ref{thm:main1}, it follows that:
\[
\I_{(\Omega,d,\mu)}(v) \geq \limsup_{m \rightarrow \infty} \inf_{H \in \Real, a \in [D_m-D_m,D_m]} \J\brac{ J_{H,\rho_m,n+q_m-1}, [-a,D_m-a] }(v) \;\;\; \forall v \in (0,1) ~.
\]
Note that if $D = \infty$ we may set $D_m = \infty$ (and that if $q = \infty$ we may also set $q_m = \infty$).
The fact that the limit and infimum on the right-hand side above may be interchanged follows directly from Lemma \ref{lem:profile-continuity} in the Appendix, thereby concluding the proof.
\end{proof}

\begin{rem}
Inspecting the proof of \cite[Proposition 6.8]{EMilman-RoleOfConvexity}, it is possible to extend the definition of the generalized CDD condition further, by allowing the densities of $\mu_m$ to only be locally uniformly bounded from above, but we do not insist on this here.
\end{rem}

\begin{rem}
It is not difficult to show that when $(M^n,g,\mu)$ satisfies a weaker form of the $CDD(\rho,n+q,D)$ condition, where we only require that $\Psi > 0$ in $\Omega$ (and not $\overline{\Omega}$), that $q (\Psi^{1/q} - 1) \in C^2(\overline{\Omega})$ (interpreted when $q=\infty$ as $\log \Psi \in C^2(\overline{\Omega})$), and that $\Omega$ satisfies either of the following conditions, then $(M^n,g,\mu)$ satisfies our generalized $CDD(\rho,n+q,D)$ condition (we omit the details):
\begin{itemize}
\item
$\Omega$ may be approximated from within by geodesically convex domains $\Omega_\eps$ with $\overline{\Omega_\eps} \subset \Omega$ and $C^2$ smooth boundaries.
\item
$\Omega$ is geodesically convex, has a $C^2$ smooth boundary, which is in addition assumed strongly convex (i.e. its second fundamental form is strictly positive definite at all points).
\end{itemize}
\end{rem}

\section{Concluding Remarks} \label{sec:remarks}

\subsection{Model Spaces for Related Problems} \label{subsec:BakryQian}

Naturally, various other geometric and analytic quantities admit one-dimensional model spaces as extremal cases under the $CDD(\rho,n+q,D)$ condition, but the collection of model spaces which are extremal for the isoperimetric problem treated here seems to be new. The two quantities most related to our work are Cheeger's constant, defined as $\inf_{v \in (0,1)} \I(v) / \min(v,1-v)$, and the spectral-gap of the Neumann Laplacian associated to the stationary measure $\mu$, namely $\Delta + g(\nabla \log \Psi, \nabla)$; these two quantities are known to be intimately connected by the works of Cheeger \cite{CheegerInq} (cf. \cite{MazyaCheegersInq1}), Buser \cite{BuserReverseCheeger} and Ledoux \cite{LedouxSpectralGapAndGeometry}.  The extremal model spaces for Cheeger's constant have been established for $\rho \leq 0$ or $D = \infty$ by Gallot \cite[Theorem 6.14]{GallotIsoperimetricInqs} (for $q=0$, extended to $q < \infty$ by Bayle \cite[Theorem E.3.3-4]{BayleThesis}), and the extremal model spaces for the spectral-gap have been established by Bakry and Qian \cite{BakryQianSharpSpectralGapEstimatesUnderCDD}, following the works of Lichnerowicz (e.g. \cite{GHLBookEdition3}), Li--Yau \cite{LiYauEigenvalues} and Zhong--Yang \cite{ZhongYangImprovingLiYau}.
However, neither of these collections of model spaces exhibits the full diversity given by Corollary \ref{cor:main1}: for those quantities, the choice of $\rho$, $n+q$ and $D$ uniquely determines a single model space (no need to go over a one-parameter family), whose corresponding one-dimensional density is in addition always symmetric about some point (w.l.o.g. the origin),  in contrast to most model spaces for the isoperimetric problem appearing in Corollary \ref{cor:main1}.

\subsection{Alternative Approaches}

In the \emph{Euclidean} setting, an alternative approach for deriving Theorem \ref{thm:main1}, which however does not extend to the full Riemannian setting (cf. \cite{EMilmanIsoperimetricBoundsOnManifolds}), is by reducing to the one-dimensional case treated in Corollary \ref{cor:main1-1D} using the localization technique of \cite{PayneWeinberger,Gromov-Milman,KLS} (e.g. as in \cite{BobkovLocalizedProofOfGaussianIso,BobkovKappaConcaveMeasures}). When $\rho \neq 0$ and $q < \infty$, this reduction seems new and of independent interest, but we leave this for a separate note. Furthermore, we may also argue that in the latter range of values for $\rho$ and $q$, the bound given in Theorem \ref{thm:main1} will \emph{not} be sharp in the Euclidean setting (as opposed to the cases $\rho = 0$ or $q = \infty$ where we could construct our model spaces in Euclidean space), so the Riemannian setting is really the right one for studying the CDD condition. 

One may also try to employ the semi-group approach pioneered by Bakry and \'Emery \cite{BakryEmery}, as in \cite{BakryLedoux,BakryQianSharpSpectralGapEstimatesUnderCDD}, but this approach crucially depends on the existence of good functional versions of the corresponding isoperimetric inequalities, and to the best of our knowledge, this is currently only known for the case $\rho > 0$ and  $q = D = \infty$ (see Bobkov \cite{BobkovGaussianIsopInqViaCube}). 

Consequently, the geometric approach we employ seems (at present) the only way of obtaining Theorem \ref{thm:main1}.

\subsection{Applications} \label{subsec:applications}

It is classical (cf. Federer--Fleming \cite{FedererFleming}, Maz'ya \cite{MazyaSobolevImbedding} and Cheeger \cite{CheegerInq}) that isoperimetric inequalities imply corresponding Sobolev inequalities, and it is also known (see Ledoux \cite{LedouxSpectralGapAndGeometry} and also \cite{EMilmanRoleOfConvexityInFunctionalInqs}) that Sobolev inequalities may be strengthened to isoperimetric inequalities under a curvature lower-bound. Consequently, as an application of the isoperimetric inequalities described in this work, we obtain in a subsequent work corresponding Sobolev inequalities on spaces satisfying the $CDD(\rho,n+q,D)$ condition. Up to numeric constants, our Sobolev inequalities are best possible, capturing the correct behaviour in $\rho$, $n+q$ and $D$, and in many cases improve the best known bounds. For instance, we are able to show that the log-Sobolev constant $\rho_{LS}$ (see \cite{Ledoux-Book} for definitions) of a space satisfying $CDD(\rho,\infty,D)$ is bounded below by:
\[
\sqrt{\rho_{LS}} \geq \frac{c}{\int_{0}^{C D} \exp\brac{-\frac{\rho}{2} t^2} dt} ~,
\]
where $c,C>0$ are numeric constants, improving in certain regimes the best known bounds by Wang \cite{WangIntegrabilityForLogSob} and Bakry--Ledoux--Qian \cite{BakryLedouxQianUnpublished}.

\section*{Appendix}
\renewcommand{\thesection}{A}
\setcounter{thm}{0}
\setcounter{equation}{0}
\setcounter{subsection}{0}

In the appendix, we prove
some useful properties of the lower bound given by Theorem \ref{thm:main1}, which are not central to the main results in this work.

\begin{lem} \label{lem:J-monotone}
For any $H \in \Real$, the function $J_{H,\rho,m}$ is pointwise monotone non-decreasing in $m \in [0,\infty]$ and monotone non-increasing in $\rho \in \Real$.
\end{lem}
\begin{proof}
The claim for $m=0$ follows by direct inspection. When $m> 0$, recall that by Remark \ref{rem:J-char}, $J_{H,\rho,m}$ coincides with the solution $J$ to:
\[
- m \frac{(J^{1/m})''}{J^{1/m}} = -(\log J)'' - \frac{1}{m} ((\log J)')^2 = \rho ~,~ J(0) = 1 ~,~ J'(0) = H ~,
\]
on the maximal interval $(a_{H,\rho,m},b_{H,\rho,m})$ containing the origin where this solution exists.
It follows that if $0 < m_1 \leq m_2 \leq \infty$ and $\rho_1 \geq \rho_2$, then on $(a_{H,\rho_1,m_1},b_{H,\rho_1,m_1})$:
\[
- m_2 \frac{(J_{H,\rho_1,m_1}^{1/m_2})''}{J_{H,\rho_1,m_1}^{1/m_2}} \geq \rho_2 = - m_2 \frac{(J_{H,\rho_2,m_2}^{1/m_2})''}{J_{H,\rho_2,m_2}^{1/m_2}} ~,~ J_{H,\rho_i,m_i}^{1/m_2}(0) = 1 ~,~ (J_{H,\rho_i,m_i}^{1/m_2})'(0) = \frac{H}{m_2} ~,~ i=1,2 ~.
\]
Consequently, an application of the maximum principle as in the proof of Lemma \ref{lem:Phi} implies that $J_{H,\rho_1,m_1}^{1/m_2} \leq J_{H,\rho_2,m_2}^{1/m_2}$ on $(a_{H,\rho_1,m_1},b_{H,\rho_1,m_1})$ and in particular that $(a_{H,\rho_1,m_1},b_{H,\rho_1,m_1}) \subset (a_{H,\rho_2,m_2},b_{H,\rho_2,m_2})$, and so the assertion follows.
\end{proof}

\begin{lem} \label{lem:profile-continuity}
The lower bound given by Theorem \ref{thm:main1}, namely the function:
\begin{equation} \label{eq:profile-continuity}
\Real \times (0,\infty] \times (0,\infty) \times (0,1) \ni (\rho,m,D,v) \mapsto \inf_{H \in \Real, a \in [D-D,D]} \J\brac{ J_{H,\rho,m}, [-a,D-a] }(v) \in (0,\infty) ~,
\end{equation}
is continuous. A similar statement holds if we fix $D = \infty$ on the domain $(0,\infty) \times (0,\infty] \times \set{\infty} \times (0,1)$.
\end{lem}
\begin{proof}
Set $F_{H,\rho,m,a}(t) = \int_{-a}^t J_{H,\rho,m}(s) ds$ and $I_{H,\rho,m,a} = J_{H,\rho,m} \circ F_{H,\rho,m,a}^{-1}$, and note that by definition:
\[
\J(J_{H,\rho,m},[-a,D-a])(v) = \min\brac{\frac{I_{H,\rho,m,a}(v F_{H,\rho,m,a}(D-a))}{F_{H,\rho,m,a}(D-a)} ,\frac{I_{H,\rho,m,a}((1-v) F_{H,\rho,m,a}(D-a))}{F_{H,\rho,m,a}(D-a)}} ~.
\]
Since $J_{H,\rho,m}(s)$ can only vanish outside some interval, it follows that $F_{H,\rho,m,a}^{-1}$ is continuous on $[0,F_{H,\rho,m,a}(+\infty)]$, where $F_{H,\rho,m,a}(+\infty) = \lim_{t \rightarrow +\infty} F_{H,\rho,m,a}(t)$, and we take the inverse at the end-points using the natural convention.
Consequently, all the functions above are continuous in their respective parameters, and so the function:
\[
(\rho,m,D,v,H,a) \mapsto \J\brac{ J_{H,\rho,m}, [-a,D-a] }(v)
\]
is continuous on the corresponding domain.

Now assume that $D < \infty$ and $v \in (0,1)$, and recall that by Corollary \ref{cor:inf2=inf1}, we know that:
\[
\inf_{H \in \Real, a \in [0,D]} \J(J_{H,\rho,m},[-a,D-a])(v) = \inf_{H \in \Real} \J(J_{H,\rho,m},[-a_H,D-a_H])(v) = \inf_{H \in \Real} \frac{1}{\int_{-a_H}^{D-a_H} J_{H,\rho,m}(t) dt} ~,
\]
where $a_H$ satisfies:
\[
\frac{v}{1-v} = \frac{\int_{-a_H}^0 J_{H,\rho,m}(t) dt}{\int_{0}^{D-a_H} J_{H,\rho,m}(t) dt} ~.
\]
Using e.g. that $J_{H,\rho,m} \leq J_{H,\rho,\infty}$ according to Lemma \ref{lem:J-monotone}, we estimate:
\begin{eqnarray*}
&  &
\lim_{H \rightarrow \infty} \int_{-a_H}^{D-a_H} J_{H,\rho,m}(t) dt =
\lim_{H \rightarrow \infty} \frac{1}{v} \int_{-a_H}^{0} J_{H,\rho,m}(t) dt \\
&\leq &
\lim_{H \rightarrow \infty} \frac{1}{v} \int_{-a_H}^{0} \exp(H t - \frac{\rho}{2} t^2) dt
\leq \lim_{H \rightarrow \infty} \frac{1}{v} \exp(-\frac{\min(\rho,0)}{2} D^2) \frac{1}{H}  = 0 ~.
\end{eqnarray*}
Similarly, we see that $\lim_{H \rightarrow -\infty} \int_{-a_H}^{D-a_H} J_{H,\rho,m}(t) dt = 0$. Moreover, note that in either case, the rate of convergence to $0$ is uniform in $m \in [0,\infty]$ and in $\rho,D,v$, as long as $\rho \in \Real$ is bounded below by $\rho_0$, $D\in (0,\infty)$ is bounded above by $D_0$, and $\min(v,1-v) \geq v_0 > 0$. It follows that:
\[
\inf_{H \in \Real} \frac{1}{\int_{-a_H}^{D-a_H} J_{H,\rho,m}(t) dt} = \inf_{H \in K_{\rho_0,D_0,v_0}} \frac{1}{\int_{-a_H}^{D-a_H} J_{H,\rho,m}(t) dt} ~,
\]
where $K = K_{\rho_0,D_0,v_0} \subset \Real$ is some compact interval depending solely on its parameters. Consequently:
\begin{eqnarray*}
&      & \inf_{H \in \Real, a \in [0,D]} \J(J_{H,\rho,m},[-a,D-a])(v) = \inf_{H \in K} \J(J_{H,\rho,m},[-a_H,D-a_H])(v) \\
& \geq & \inf_{H \in K, a \in [0,D]} \J(J_{H,\rho,m},[-a,D-a])(v) \geq \inf_{H \in \Real, a \in [0,D]} \J(J_{H,\rho,m},[-a,D-a])(v) ~,
\end{eqnarray*}
and so we conclude we must have equality signs everywhere. It follows that it is enough to test the infimum in (\ref{eq:profile-continuity}) on the compact set $K_{\rho_0,D_0,v_0} \times [0,D]$. By compactness, the function $(\rho,m,D,v,H,a) \mapsto  \J\brac{ J_{H,\rho,m}, [-a,D-a] }(v)$ is continuous, uniformly in $(H,a) \in K_{\rho_0,D_0,v_0} \times [0,D]$, and so the infimum over this set is continuous in $(\rho,m,D,v)$ in the corresponding domain, as asserted. The case when we fix $D=\infty$ is treated similarly.
\end{proof}

An immediate corollary of the above proof is that (the cases $v \in \set{0,1}$, $m=0$, $D = \infty$ and $\rho \leq 0$ below hold trivially):
\begin{cor} \label{cor:inf-attained}
For any $\rho \in \Real$, $m \in [0,\infty]$, $D \in (0,\infty]$ and $v \in [0,1]$, the infimum in the lower bound given by Theorem \ref{thm:main1} is attained:
\[
\inf_{H \in \Real, a \in [D-D,D]} \J\brac{ J_{H,\rho,m}, [-a,D-a] }(v) = \min_{H \in \Real, a \in [D-D,D]} \J\brac{ J_{H,\rho,m}, [-a,D-a] }(v) ~.
\]
\end{cor}

\setlinespacing{0.88}
\setlength{\bibspacing}{2pt}

\bibliographystyle{plain}
\bibliography{../../ConvexBib}

\end{document}